\documentclass[10pt]{amsart}
\numberwithin{equation}{section}
\usepackage{amsmath,amssymb,amsthm}
\usepackage{tikz}
\usepackage{color}
\usepackage{tikz-cd}
\usetikzlibrary{arrows,matrix}
\usetikzlibrary{shapes.geometric}
\usetikzlibrary{decorations.markings} 
\usepackage[textsize=small]{todonotes}
\usepackage{xparse}
\usepackage{xspace}
\usepackage{lipsum}
\usepackage{mathtools}
\usepackage[all]{xy}
\mathtoolsset{showonlyrefs}
\usepackage{hyperref}

\tikzset{
  vertice/.style={circle,draw=black},
  decoration={markings,mark=at position 0.5 with {\arrow{>}}}
}

\usepackage[vcentermath,enableskew]{youngtab}
\let\cal\mathcal

\def\Fscr{{\cal F}}

\def\Oscr{{\cal O}}

\let\blb\mathbb

\def \ZZ{{\blb Z}}

\def \NN{{\blb N}}

\def\Id{\operatorname{id}}

\def\Rep{\operatorname {Rep}}
\def\GL{\operatorname {GL}}

\def\Ext{\operatorname {Ext}}
\def\Hom{\operatorname {Hom}}

\def\End{\operatorname {End}}

\def\im{\operatorname {im}}

\def\ker{\operatorname {ker}}

\def\End{\operatorname {End}}

\DeclareMathOperator{\Ind}{Ind}

\DeclareMathOperator{\uaut}{\underline{aut}}
\DeclareMathOperator{\uend}{\underline{end}}
\DeclareMathOperator{\image}{Im}
\DeclareMathOperator{\eval}{ev}
\DeclareMathOperator{\coeval}{coev}
\DeclareMathOperator{\induced}{ind}
\DeclareMathOperator{\socle}{soc}
\DeclareMathOperator{\topp}{top}

\newtheorem{lemma}{Lemma}[section]
\newtheorem{proposition}[lemma]{Proposition}
\newtheorem{theorem}[lemma]{Theorem}
\newtheorem{corollary}[lemma]{Corollary}

\theoremstyle{definition}

\newtheorem{example}[lemma]{Example}
\newtheorem{definition}[lemma]{Definition}

{

}

\theoremstyle{remark}

\newtheorem{remark}[lemma]{Remark}

\newdimen\uboxsep \uboxsep=1ex
\def\uboxn#1{\vtop to 0pt{\hrule height 0pt depth 0pt\vskip\uboxsep
\hbox to 0pt{\hss #1\hss}\vss}}

\def\uboxs#1{\vbox to 0pt{\vss\hbox to 0pt{\hss #1\hss}
\vskip\uboxsep\hrule height 0pt depth 0pt}}

\def\tn{\textnormal}

\def\z{\mathbb{Z}}
\def\n{\mathbb{N}}

\def\ot{\otimes}

\def\oo{\Oscr_{\tt nc}(G)}
\def\ll{\lambda}

\usepackage{xcolor}
\hypersetup{
    colorlinks,
    linkcolor={red!50!black},
    citecolor={blue!50!black},
    urlcolor={blue!80!black}
}

\newcounter{todocounter}
\DeclareDocumentCommand\addreference{g}{\stepcounter{todocounter}\todo[color
  = blue!30, fancyline]{\thetodocounter. Add
    reference\IfNoValueF{#1}{: #1}}\xspace}
\DeclareDocumentCommand\checkthis{g}{\stepcounter{todocounter}\todo[color
  = red!50, fancyline]{\thetodocounter. Check this\IfNoValueF{#1}{:
      #1}}\xspace}
\DeclareDocumentCommand\fixthis{g}{\stepcounter{todocounter}\todo[color
  = orange!50, fancyline]{\thetodocounter. Fix this\IfNoValueF{#1}{:
      #1}}\xspace}
\DeclareDocumentCommand\expand{g}{\stepcounter{todocounter}\todo[color
  = green!50, fancyline]{\thetodocounter. Expand\IfNoValueF{#1}{:
      #1}}\xspace}
\DeclareDocumentCommand\remarkk{g}{\stepcounter{todocounter}\todo[color
  = yellow!50, fancyline]{\thetodocounter. Remark\IfNoValueF{#1}{:
      #1}}\xspace}

\title[The representation theory of
  non-commutative  $\Oscr(\GL_2)$]{The representation theory of
  non-commutative \boldmath $\Oscr(\GL_2)$}
\author{Theo Raedschelders and Michel Van den Bergh}
\address{(Theo Raedschelders) \newline Departement Wiskunde, Vrije Universiteit Brussel, 
Pleinlaan $2$, B-1050 Elsene \newline E-mail address: {\tt traedsch@vub.ac.be}}
\address{(Michel Van den Bergh) \newline Departement WNI, Universiteit Hasselt, Universitaire Campus \\
B-3590 Diepenbeek \newline E-mail address: {\tt michel.vandenbergh@uhasselt.be}}
\let\oldmarginpar\marginpar
\def\marginpar#1{\oldmarginpar{\color{red}\tiny #1}}

\begin{document}
\maketitle
\begin{abstract}
  In our companion paper ``The {M}anin {H}opf algebra of a {K}oszul
  {A}rtin-{S}chelter regular algebra is quasi-hereditary'' we used the
  Tannaka-Krein formalism to study the universal coacting Hopf algebra
  $\uaut(A)$ for a Koszul Artin-Schelter regular algebra $A$. In this
  paper  we study in detail the case $A=k[x,y]$. In
  particular we give a more precise description of the standard and
  costandard representations of $\uaut(A)$ as a coalgebra and we show that the latter can be
  obtained by induction from a Borel quotient algebra.  Finally we
  give a combinatorial characterization of the simple $\uaut(A)$-representations as tensor
  products of $\uend(A)$-representations and their duals.
\end{abstract}

\setcounter{tocdepth}{1}
{\hypersetup{linkcolor=black}
\tableofcontents
}

\newcommand\blfootnote[1]{%
  \begingroup
  \renewcommand\thefootnote{}\footnote{#1}%
  \addtocounter{footnote}{-1}%
  \endgroup
}

\blfootnote{The first author is an Aspirant at the FWO.} 
\blfootnote{The second author is a senior researcher at the FWO.}

\section{Introduction}
\label{section:introduction}

In \cite{manin}  Manin
constructs for any graded algebra $A=k\oplus A_1\oplus A_2\oplus\cdots$ 
   a bialgebra~$\uend(A)$ and a Hopf algebra~$\uaut(A)$ coacting on it in a universal way.
The Hopf algebra $\uaut(A)$ should be thought of as the non-commutative
symmetry group of~$A$.

The representation theory of the bialgebra $\uend(A)$ was fully
described in \cite{kriegk-vandenbergh} in the case that $A$ is a Koszul algebra.
In our recent paper \cite{raedschelders-vandenbergh-2} we extended this to~$\uaut(A)$
when~$A$ is in addition Artin-Schelter regular \cite{AS}. We show in particular
that~$\uaut(A)$ is quasi-hereditary as a coalgebra and we give a description of its monoidal
category of comodules. The methods in loc.\ cit.\ are based on Tannakian
duality and  are fairly agnostic to the specific choice of $A$.


On the other hand when $A=k[x_1,\ldots,x_d]$ it is reasonable to
think of $\uaut(A)$ as some sort of non-commutative coordinate ring  
of $\GL_n$. From this point of view
one may hope that techniques from the theory of algebraic groups would yield extra insight into
the representation theory of $\uaut(A)$.
Obvious examples  of such techniques are highest weight theory and induction from
Borel subgroups, but more
combinatorial approaches  based on standard monomial
theory and straightening laws are also useful to keep in mind. 

In this paper we discuss the most basic case, namely $A=k[x,y]$.
We will write~$\Oscr_{\tt nc}({{\GL_2}})$ for $\uaut(A)$ to emphasize the
fact that we view the latter as a non-commutative variant of the algebraic group ${{\GL_2}}$.

As an algebra $\Oscr_{\tt nc}({{\GL_2}})$ is generated by the entries of the matrix
(see \S\ref{sec:genrel})
\[
M=\begin{pmatrix}
a&b\\
c&d
\end{pmatrix}
\]
together with the formal inverse of the determinant $\delta:=ad-cb$.
The following additional relations are imposed:
\[
\begin{aligned}
ac-ca & =  0\\
bd-db &= 0 \\ 
ad-cb & =   da-bc\\
a\delta^{-1}d-b\delta^{-1}c & =  1  =  d\delta^{-1}a-c\delta^{-1}b,\\
b\delta^{-1}a-a\delta^{-1}b & =  0  =  c\delta^{-1}d-d\delta^{-1}c
\end{aligned}
\]
The bialgebra structure on $\Oscr_{\tt nc}({{\GL_2}})$ is given by $ \Delta(M)=M\otimes M $.  
The
first three equations express that $M$ is a ``Manin matrix''
\cite{chevrov-falqui-rubtsov}. The last four equations are forced upon us by the requirement that
$\Oscr_{\tt nc}({{\GL_2}})$ must have an antipode.

It follows from \cite{raedschelders-vandenbergh-2} that the coalgebra $\Oscr_{\tt nc}({{\GL_2}})$ is
quasi-hereditary. In the current paper we will give a proof of
this fact which is  different in spirit from the general one in~\cite{raedschelders-vandenbergh-2}. In particular
we will obtain more explicit descriptions of the
(co)standard and the simple comodules that come with the quasi-hereditary structure. The reader
not familiar with quasi-hereditary (co)algebras may consult \S\ref{section:finite-dimensional-quasihereditary},\S\ref{sec:infqh} for a short introduction
and further references. Here we will content ourselves with noting
that the fact that $\Oscr_{\tt nc}({{\GL_2}})$ is quasi-hereditary immediately implies that it has a large
number of standard representation theoretic properties which are reminiscent of the representation theory of reductive groups.

We now give a more precise description of our results. Let $\Lambda$ be the monoid $\langle d,\delta^{\pm1}\rangle$ ($d,\delta$ are used as formal symbols here). We equip~$\Lambda$ with the left and right invariant ordering generated by
$
1<d\delta^{-1}d$, $\delta<dd
$.
In addition we equip $\Lambda$ with an order preserving duality given by
$d^\ast=d\delta^{-1}$, $\delta^\ast=\delta^{-1}$, $(\lambda\mu)^\ast=\mu^\ast\lambda^\ast$.  For
$\lambda=\delta^{x_1} d^{y_1} \cdots \delta^{x_n} d^{y_n}$ we let
$\nabla(\lambda)$ be the subcomodule of the regular comodule
$\Oscr_{\tt nc}({{\GL_2}})$ spanned by vectors $
\delta^{x_1}b^{y'_1}d^{y''_1} \cdots \delta^{x_n} b^{y'_n}d^{y''_n}, $
where $y'_i+y''_i=y_i$. We also put $\Delta(\lambda^*)$ for
$\nabla(\lambda)^\ast$. 

By construction $\nabla(\lambda)$ contains the vector $\lambda$. Let $L(\lambda)$ be the subcomodule of
$\nabla(\lambda)$ cogenerated by $\lambda$. The following is our first main result.
\begin{theorem}[Proposition \ref{llll}, Theorem \ref{th:qh}]
\label{referentie}
The coalgebra $\Oscr_{\tt nc}({{\GL_2}})$ is quasi-hereditary with respect to the poset $(\Lambda,\le)$.
The standard, costandard and simple comodules are given by $\Delta(\lambda)$, $\nabla(\lambda)$ and $L(\lambda)$
as introduced above. 
\end{theorem}
In Lemma \ref{basis} we give an explicit basis for $\Oscr_{\tt nc}({{\GL_2}})$ obtained via the Bergman diamond lemma.
The ``spanning set'' for $\nabla(\lambda)\subset \Oscr_{\tt nc}({{\GL_2}})$ we have given is actually part of the 
basis of $\Oscr_{\tt nc}({{\GL_2}})$. 
In the process of proving the quasi-hereditary  property  we have to verify
 that $\Oscr_{\tt nc}({{\GL_2}})$ has a $\nabla$-filtration. Roughly speaking we do
this by comparing the explicit bases for~$\nabla(\lambda)$ and  $\Oscr_{\tt nc}({{\GL_2}})$. This approach is different
from \cite{raedschelders-vandenbergh-2}. 

In the commutative case the costandard representations are sometimes called dual Weyl modules \cite{jantzen2}
and they are obtained by induction from one-dimensional representations of a Borel subgroup. It is natural to try
to imitate this construction in the non-commutative case. 

To do so we  define the following quotient Hopf algebras of $\Oscr_{\tt nc}({{\GL_2}})$:
\[
\begin{aligned}
\Oscr_{\tt nc}(B)&=\Oscr_{\tt nc}({{\GL_2}})/(b) \cong k \langle c,d^{\pm 1} \rangle [a^{\pm 1}], \\
\Oscr(T)&=\Oscr_{\tt nc}({{\GL_2}})/(b,c) \cong  k[a^{\pm 1},d^{\pm 1}]
\end{aligned}
\]
Here $\Oscr(T)$ is the actual commutative coordinate ring of a two-dimensional torus~$T$. We identify
its character group $X(T)$  with the Laurent monomials in $a,d$.
By sending $\delta\in \Lambda$ to $ad\in X(T)$ and $d\in \Lambda$ to $d\in X(T)$ we obtain a map
of monoids ${\tt wt}:\Lambda\rightarrow X(T)$.

\def\comod{\operatorname{CoMod}}
If $t\in X(T)$ then there is an associated one-dimensional $\Oscr(T)$-representation $k_t$
which may also be viewed as a $\Oscr_{\tt nc}(B)$-representation. Denote by $\Ind^{{\GL_2}}_B $ the
right adjoint to the restriction functor $\comod(\Oscr_{\tt nc}({{\GL_2}}))\rightarrow \comod(\Oscr_{\tt nc}(B))$
(see \S\ref{sec:induced}). Then we have the following result: 
\begin{theorem}[Theorem \ref{induced}] One has
\[
	\induced_B^{{\GL_2}}(k_t) = \bigoplus_{\substack{\lambda \in \Lambda \\ {\tt wt}(\lambda)=t}} \nabla(\lambda).
\]
\end{theorem}
In particular we see that $\induced^{{\GL_2}}_B(k_t)=0$ if $t\not\in X(T)^+:=\im {\tt wt}$. This agrees with  the commutative case where only dominant weights yield non-zero representations under induction. But we also
see that in contrast to the commutative case here the induced representations are not indecomposable. However
they still yield all costandard comodules. 

In the commutative case the higher derived induction functors
$R^i\Ind^{G}_B$ are the subject of deep results such as Kempf's
vanishing theorem and more generally (in characteristic zero) Bott's
theorem. It would be interesting to know if such results also exist in
the non-commutative case.  We hope to come back to this in the
future.

From the fact that  $\Oscr_{\tt nc}({{\GL_2}})$ is quasi-hereditary it follows by general theory that the simple comodules are of the form $L(\lambda)=\im(\Delta(\lambda)\rightarrow \nabla(\lambda))$
which in principle reduces their study
to a linear algebra problem. 

 This  problem is usually difficult to solve,
but fortunately we succeed in the special case we are considering. 
The bialgebra
$\Oscr_{\tt nc}(M_2):=\uend(A)$ is the subalgebra\footnote{The fact that it is a subalgebra follow from Lemma \ref{basis} or else by \cite[Corollary 6.3.5]{raedschelders-vandenbergh-2}.} of $\Oscr_{\tt nc}({{\GL_2}})$ generated by $a,b,c,d$
and we have:
\begin{theorem}[Theorem \ref{th:simplelist}, Corollary \ref{cor:simple}] Assume that $k$ has characteristic zero. All simple
  $\Oscr_{\tt nc}(\GL_2)$-representations are repeated tensor products
  of simple 
 $\Oscr_{\tt nc}(M_2)$-representations and their duals.
\end{theorem}
The characteristic zero hypothesis is likely superfluous. It comes from the fact that we use some fragments of the representation theory of (commutative) $\GL_2$ in the proof.

The simple representations of $\Oscr_{\tt nc} (M_2)$  were
classified in \cite{kriegk-vandenbergh}. They are tensor products of $(S^n V)_{n\in
  \NN}$ and $\wedge^2 V$, where $V$ denotes the standard representation. Thus every simple $\Oscr_{\tt nc}(\GL_2)$-representation is a
tensor product of these basic representations and their duals. 
It is slightly subtle to characterize which among those tensor products are simple.
This is done in Theorem
\ref{th:simplelist}. Note that the problem of finding explicit models for the irreducible representations of universal quantum groups, in connection with Borel-Weil theory, was already raised in~\cite{wang}. 

For people interested in universal quantum groups we refer to \cite{bichon-dubois-violette,chirvasitu,walton-wang}
for some other recent papers on this subject. In particular  \cite{walton-wang} discusses certain quotients
of $\uaut(A)$ (which the authors denote by $\Oscr_A(\GL)$) when $A$ is a two-dimensional Artin-Schelter regular algebra.  The emphasis in loc.\ cit.\ is on the algebra 
properties of these quotients so the results are more or less orthogonal to the ones contained in
this paper. Note however that certain properties of bialgebras, like their Hilbert series, can be studied both on the algebra
and on the coalgebra side. 

Finally, note that in~\cite{bichon-riche} the authors also study representations of certain universal quantum groups by relying on a Borel-Weil type construction. It can however be checked that $\uaut(A)$ does not fit into their axiomatic framework since it does not have a ``dense big cell''. Our paper (see also~\cite{raedschelders-vandenbergh-2}) partially meets their lack of a ``non-commutative root system'' by providing natural orderings on the non-commutative weight monoid $\Lambda$, compatible with the one on $\GL_2$.

\section{Preliminaries}
\label{section:preliminaries}
Let~$k$ denote an algebraically closed field. 
All coalgebras $C$ are $k$-coalgebras and all unadorned
tensor products are over $k$.  By default a $C$-comodule $V$ is a left
comodule, i.e.\ with structure map $V \to C \ot V$. By a
$C$-representation we mean a finite dimensional $C$-comodule.  We
refer to Green~\cite{green} for fundamental facts and proofs on
coalgebra representation theory.

A beautiful survey on the use of quasi-hereditary (co)algebras in the
representation theory of algebraic groups is given by Donkin
in~\cite{donkin} and we will use the main definitions from that
article. One should also mention Jantzen's book on algebraic
groups~\cite{jantzen2} which contains all the essential results but
does not use the quasi-hereditary formalism. Finally for an algebraic
study of quasi-hereditary algebras we refer to the classic paper by
Dlab and Ringel~\cite{dlab-ringel-2}.  The reader should be warned that
the basic definitions in~\cite{dlab-ringel-2} are different from those of
\cite{donkin}. For a comparison see Appendix \ref{appA}.
\subsection{Finite dimensional quasi-hereditary coalgebras}
\label{section:finite-dimensional-quasihereditary}
In this section we follow~\cite{donkin}. 
Assume~$C$ is a finite dimensional coalgebra and let $\{L(\lambda) \
\vert \ \lambda \in \Lambda \}$ be a complete set of non-isomorphic
simple $C$-comodules for some partially ordered set $(\Lambda,\leq)$. By $I(\lambda)$ we denote the injective hull
of the simple comodule $L(\lambda)$. Let $V$ be a $C$-represenation.
For $\pi \subset \Lambda$ we say that $V$ belongs to $\pi$ if all composition factors
of $V$ are in the set $\{L(\lambda) \ \vert \ \lambda \in \pi \}$. In general we write
$O_\pi(V)$ for the comodule that is maximal amongst all subcomodules of $V$ belonging to
$\pi$.  For $\lambda\in \Lambda$ put $\pi(\lambda)=\{\mu \in \Lambda \ \vert \
\mu < \lambda\}$. Then $\nabla(\lambda) \supset L(\lambda)$ is the
subcomodule of $I(\lambda)$ defined by
\begin{equation}
\label{definition:donkinco}
\nabla(\lambda)/L(\lambda)=O_{\pi(\lambda)}(I(\lambda)/L(\lambda))
\end{equation}
The $\nabla(\ll)$ are called costandard comodules. 
Using the notation
$O^{\pi}(V)$ to denote the minimal subcomodule $U$ of $V$ such that
$V/U$ belongs to $\pi$, the standard comodules $\Delta(\lambda)$ are defined dually as
$$
\Delta(\ll)=P(\ll)/O^{\pi(\ll)}(N(\ll))
$$
where $P(\ll)$ is the projective cover of $L(\ll)$ and $N(\ll)$ denotes its maximal proper subcomodule.

From the definitions, one has more or less immediately the following proposition.	

	\begin{proposition}
	\label{simples}
	One has  
	$$
	\Hom^C(\Delta(\ll),\nabla(\mu))=
	\begin{cases}
		 k & \text{if } \lambda=\mu\\ 
		 0 & \text{otherwise}
	\end{cases}, 
	$$
	and all the simples can be recovered as $L(\ll)=\tn{Im}(\Delta(\ll) \to \nabla(\ll))$.
	\end{proposition}
        To verify that a comodule is costandard, we will use the
        following coalgebraic version of Lemma~$1.1$
        in~\cite{dlab-ringel-2} (slightly adapted to be correct for the
        setting from \cite{donkin} we are following).

\begin{lemma}
\label{lemma:donkincostandard}
For any~$C$-comodule~$V$, and~$\lambda \in \Lambda$, the following are equivalent:
\begin{enumerate}
\item $V \cong \nabla(\lambda)$,
\item the following three conditions are satisfied:
\begin{enumerate}
\item $\socle(V) \cong L(\lambda)$,
\item if~$[V/\socle(V):L(\mu)] \neq 0$, then~$\mu < \lambda$,
\item if~$\mu < \lambda$, then $\Ext^1(L(\mu),V)=0$.
\end{enumerate}
\end{enumerate}
\end{lemma}

Let~$G_0(C)$ denote the Grothendieck group of the category of finite dimensional~$C$-comodules.
\begin{lemma}
\label{grothendieck}
The (co)standard comodules form a $\mathbb{Z}$-basis of $G_0(C)$.
\end{lemma}
\begin{proof}
The simple $L(\lambda)$ occurs with multiplicity $1$ in $\nabla(\ll)$, and by definition all other composition factors of $\nabla(\ll)$ are of strictly smaller weight so the costandard comodules are related to the basis of simple comodules by a unitriangular matrix. The proof for the standard comodules is similar.
\end{proof}
By $\Fscr(\Delta)$, $\Fscr(\nabla)$ one denotes the categories of representations admitting filtrations whose
factors are respectively standard and costandard modules. We will call such filtrations (co)standard filtrations (they are required to exist but are not part of the structure of an object in $\Fscr(\Delta)$, $\Fscr(\nabla)$).

Note that Lemma~\ref{grothendieck} ensures that the multiplicity $[V:\nabla(\ll)]$ of $\nabla(\ll)$ as subquotient in a costandard filtration on $V$ is independent of the filtration.

	\begin{definition} \cite{donkin}
	\label{qh}
	The (finite dimensional) coalgebra $C$ is quasi-hereditary if 
	\begin{enumerate}
		\item $I(\lambda) \in \Fscr(\nabla)$,
		\item $(I(\ll):\nabla(\ll))=1$,
		\item If $(I(\ll):\nabla(\mu)) \neq 0$, then $\mu \geq \ll$.
	\end{enumerate}
	\end{definition}

In the following we will use another characterization of quasi-hereditary coalgebras. It is often more convenient since when combined with Lemma \ref{lemma:donkincostandard}(2) 
it does not explicitly refer to the injectives $I(\lambda)$. 
\begin{proposition}
\label{prop:convenient} 
The coalgebra $C$ is quasi-hereditary if and only if the following conditions hold.
\begin{enumerate}
\item $C\in \Fscr(\nabla)$.
\item If $\Ext^1(L(\mu),\nabla(\lambda))\neq 0$ then $\mu >\lambda$. 
\end{enumerate}
\end{proposition}
For a proof see Appendix \ref{appA}.

For use below put
$C(\pi)=O_\pi(C)$. From the maximality it follows that $C(\pi)$ is a
subcoalgebra of $C$ and that $\{L(\lambda) \ \vert \ \lambda \in \pi
\}$ is a complete set of non-isomorphic simple
$C(\pi)$-comodules. For $\lambda\in\pi$ we write $\Delta_\pi(\lambda)$, $\nabla_\pi(\lambda)$
for the corresponding $C(\pi)$-(co)standard comodules. 
A subset $\pi\subset \Lambda$ is
said to be saturated 
if $\mu \leq \ll \in \pi$ implies $\mu \in \pi$.
Recall the following 
\begin{theorem}  \label{th:truncation} \cite[Prop. A.3.4]{donkin2}
Assume that $C$ is quasi-hereditary. For a saturated subset $\pi\subset \Lambda$ we have
that $C(\pi)$ is quasi-hereditary with simple, standard and costandard modules respectively
given by $L(\lambda)$, $\Delta_\pi(\lambda)=\Delta(\lambda)$, $\nabla_\pi(\lambda)=\nabla(\lambda)$ for $\lambda\in \pi$.
\end{theorem}

\subsection{Infinite dimensional quasi-hereditary coalgebras}
\label{sec:infqh}
Since coordinate rings of algebraic groups and their quantum versions
are infinite dimensional, Definition~\ref{qh} needs to be generalized. In
this section,~$C$ is no longer assumed to be finite dimensional. In
agreement with the notation of
Section~\ref{section:finite-dimensional-quasihereditary}, let
$\{L(\ll) \ \vert \ \ll \in \Lambda\}$ denote a complete set of
non-isomorphic simple comodules of~$C$, indexed by some (possibly
infinite) poset $(\Lambda,\leq)$. By $I(\ll)$ we still denote the
injective hull of the simple comodule $L(\ll)$. Note that projective
covers in general no longer exist and hence the situation is no longer
self dual.
The following infinite dimensional version of the quasi-hereditary
property is due to Donkin~\cite{donkin}. 
	\begin{definition}
	\label{strong}
	The coalgebra $C$ is quasi-hereditary if
\begin{enumerate}
\item for every $\lambda\in \Lambda$ the set $\pi(\lambda)$ is finite;
\item for every finite, saturated $\pi \subset \Lambda$, the coalgebra
  $C(\pi)$ (see
  Section~\ref{section:finite-dimensional-quasihereditary}, the definition makes sense in the current setting) is finite
  dimensional and quasi-hereditary in the sense of
  Definition~\ref{qh}.
\end{enumerate}
	\end{definition}
Assume that $C$ is quasi-hereditary.
For $\lambda\in \Lambda$ and $\pi$ a saturated subset in $\Lambda$ containing $\lambda$ (e.g.\ $\pi(\lambda)$)
we put
\begin{align*}
\Delta(\lambda)&=\Delta_{\pi}(\lambda)\\
\nabla(\lambda)&=\nabla_{\pi}(\lambda)\,.
\end{align*}
Theorem \ref{th:truncation} shows that this definition is independent of $\pi$.
\begin{remark}
It is not hard to see that $\nabla(\ll)$ is isomorphic to the subcomodule $\nabla'(\ll)$ of $I(\lambda)$ containing $L(\lambda)$ defined by 
$$
\nabla'(\ll)/L(\ll)=O_{\pi(\lambda)}(I(\lambda)/L(\ll)),
$$
exactly like in the finite dimensional setting.  Due to the lack of projective covers in the infinite dimensional case, there is no analogous construction
for standard comodules.
\end{remark}

By $\Fscr(\nabla)$ (respectively $\Fscr(\Delta)$), we again denote the
category of representations of~$C$ having a (finite) filtration by costandard 
(respectively standard) comodules.

The following
theorem by Donkin \cite[Thm 2.5]{donkin}
shows that the homological algebra of quasi-hereditary
coalgebras is completely determined by that of their finite dimensional
quasi-hereditary subcoalgebras. 
	\begin{theorem}
	If $C$ is quasi-hereditary, then for a finite, saturated $\pi \subset \Lambda$, and $C(\pi)$-comodules $V$ and $W$, one has for all $i \geq 0$,
	\[
	\tn{Ext}_{C(\pi)}^i(V,W) \cong \tn{Ext}_C^i(V,W).
	\]
	\end{theorem}
\section{Universal coacting bialgebras and Hopf algebras}
\label{section:universal}
Throughout $A=k\oplus A_1\oplus A_2\oplus$ is an $\NN$-graded algebra such that $\dim A_i<\infty$ for
all $i$.
We first introduce the universal coacting
bialgebra $\uend(A)$ which is defined using a suitable universal property. Every
bialgebra has a universal associated Hopf algebra, which in the case
of $\uend(A)$ will be denoted $\uaut(A)$. This Hopf algebra also
satisfies a universal property and is in fact the universal coacting
Hopf algebra of $A$. Finally, we describe by generators and relations
the specific bialgebra and Hopf algebra we are interested in, namely
$\uend(k[x,y])$ and $\uaut(k[x,y])$.

\subsection{Universal constructions}

	\begin{definition}
	\label{universal}
	The universal coacting algebra of $A$, denoted $\uend(A)$, is an algebra equipped with an 	
	algebra morphism $\delta_A:A \to \uend(A) \ot A$, satisfying the following universal property: 
	for any $k$-algebra $B$ and algebra morphism $f: A \to B \ot A$, such that 
	$\delta(A_n) \subset B \ot A_n$, there exists a unique morphism 
	$g:\uend(A) \to B$ such that the diagram
	$$
	\begin{tikzcd}
		A \drar[swap]{f} \rar{\delta_A} & \uend(A) \ot A \dar[dashed]{g \ot 1}\\
		& B \ot A 
	\end{tikzcd}
	$$
	commutes.
	\end{definition}

The existence of this algebra is essentially due to Manin~\cite{manin}. These algebras have some nice properties, the proofs of which can be found in Proposition 1.3.8 of~\cite{pareigis-2}.
\begin{definition} Let $B$ be a bialgebra. A $B$-comodule algebra is an algebra $A$ equipped with
an algebra morphism $f:A\rightarrow B\otimes A$ which makes $A$ into a comodule over $B$.
\end{definition}
	\begin{proposition}
	\label{comodulealgebra}
\begin{enumerate}
\item
	The universal coacting algebra of $A$ is in fact a bialgebra, $A$ is an $\uend(A)$-comodule 
	algebra via $\delta_A$. 
\item $\uend(A)$ also satisfies a different universal property:
if $B$ is any bialgebra, and $f:A \to B \otimes A$ equips $A$ with the structure
	of a $B$-comodule algebra such that $f(A_n)\subset B\otimes A_n$, then there is a unique morphism of bialgebras $g:\uend(A) \to B$ such 
	that the diagram
	$$
	\begin{tikzcd}
		A \drar[swap]{f} \rar{\delta_A} & \uend(A) \otimes A \dar[dashed]{g \otimes 1} \\
		& B \otimes A
	\end{tikzcd}
	$$
	commutes.
\end{enumerate}
	\end{proposition}
	
The bialgebra $\uend(A)$
turns out to have a very nice representation theory when $A$ is Koszul. It was studied by the second author and B.~Kriegk in~\cite{kriegk-vandenbergh} and forms part of the motivation for this work.

Every bialgebra has a Hopf envelope, as proven by Takeuchi~\cite{takeuchi}. A detailed proof of the following theorem can be found in Pareigis~\cite{pareigis-2} (see Theorem 2.6.3).

	\begin{theorem}
	\label{hopf}
	Let $B$ be a bialgebra. Then there exists a Hopf algebra $H(B)$, called the Hopf envelope of $B$, 
	and a homomorphism of bialgebras $i:B \to H(B)$ such that for every Hopf algebra $H$ and for every 
	homomorphism of bialgebras $f:B \to H$, there is a unique homomorphism of Hopf algebras 
	$g:H(B) \to H$ such that the diagram
	$$
	\begin{tikzcd}
		B \drar[swap]{f} \rar{i} & H(B) \dar[dashed]{g} \\
		& H 
	\end{tikzcd}
	$$
	commutes.
	\end{theorem}
	
\begin{remark} \label{rem:Hopf}
The construction of $H(B)$ from $B$ is as follows: we freely adjoin to $B$ (as an algebra) variables ${s}^n(b)$ for $n\ge 1$, $b\in B$ 
and we impose the following relations
	\begin{enumerate}
	\item For $\lambda_1,\lambda_2\in k$, $b_1,b_2\in B$: ${s}^n(\lambda_1b_1+\lambda_2b_2)=\lambda_1 {s}^n(b_1)+\lambda_2 {s}^n(b_2)$.
	Furthermore ${s}^n(1)=1$.
	\item Let $a,b\in B$. If $n$ is even then ${s}^n(ab)={s}^n(a){s}^n(b)$ and if $n$ is odd then ${s}^n(ab)={s}^n(b){s}^n(a)$.
	\item For all $b\in B$: $\sum_b {s}^{n+1}(b_{(1)}){s}^n(b_{(2)})=\epsilon(b)$, 
	$\sum_b {s}^n(b_{(1)}) {s}^{n+1}(b_{(2)})=\epsilon(b)$.
	\end{enumerate}
The resulting algebra $H(B)$ is made into Hopf algebra by 
defining the coproduct, counit and antipode on $B$ as follows (with $n \geq 0$, where we identify $s^0(b)$ with $b$)
\begin{align*}
\epsilon(s^n(b))&=\epsilon(b)\\
\Delta({s}^n(b))&=
\begin{cases}
{s}^n(b_{(1)})\otimes {s}^n(b_{(2)})&\text{if $n$ is even}\\
{s}^n(b_{(2)})\otimes {s}^n(b_{(1)})&\text{if $n$ is odd}
\end{cases}\\
S({s}^n(b))&={s}^{n+1}(b)
\end{align*}
A  computation shows that these definitions are compatible with the relations we have imposed.
\end{remark}

        We will denote the Hopf envelope of $\uend(A)$ by
        $\uaut(A)$. Using Definition~\ref{universal}, there is a
        morphism of algebras $\delta_A:A \to \uaut(A) \otimes A$ such
        that $A$ is a comodule-algebra over $\uaut(A)$. This easily
        gives the final universal property.

	\begin{corollary}
	If $H$ is a Hopf algebra and $A$ is an $H$-comodule algebra by $f:A \to H \otimes A$ such that $f(A_n)\subset H\otimes A_n$, then there is a 
	unique morphism of Hopf algebras $g:\uaut(A) \to H$ such that the diagram 
	$$
	\begin{tikzcd}
	A \drar[swap]{f} \rar{\delta_A} & \uaut(A) \otimes A \dar[dashed]{g \otimes 1} \\
	& H \otimes A
	\end{tikzcd}
	$$
       commutes.
	\end{corollary}
	\begin{proof}
	First use the universal property of Proposition \eqref{comodulealgebra} to get a morphism 
	$g':\uend(A) \to H$, and then use the one of Proposition \eqref{hopf} to get a map $g$.
	\end{proof}
Following this corollary we call  $\uaut(A)$ the universal
        coacting Hopf algebra on $A$.
\subsection{Generators and relations} 
\label{sec:genrel}
In the rest of this paper we will concentrate on the first non-trival case $A=k[x,y]$ with the grading given by $|x|=|y|=1$.
In Section $5$ of~\cite{manin}, Manin shows that a (finite!) presentation of $\uend(A)$ is given by:
$$
\uend(k[x,y])=\frac{k \langle a,b,c,d \rangle}{I},
$$
where $I$ is the ideal generated by the relations:
$$
\begin{array}{lllll}
ac -ca & = & 0, \\
ad-cb & = & da-bc, \\
bd -db & = & 0.
\end{array}
$$
Denoting by $M$ the generator matrix, i.e.
$$
M = \begin{pmatrix}
a & b \\
c & d
\end{pmatrix},
$$
the bialgebra structure is given by 
$$
\Delta(M)=M \otimes M,
\epsilon(M)
=
\textnormal{Id},
$$
where Id denotes the identity matrix. 
Since there is a bialgebra epimorphism
\begin{align}
\label{surjection}
\uend(k[x,y]) \twoheadrightarrow \Oscr(M_2)
\end{align}
to the coordinate ring of the reductive algebraic monoid $M_2$, we use the notation $\Oscr_{\tt nc}(M_2)$ for $\uend(k[x,y])$.

The universal coacting Hopf algebra $\uaut(k[x,y])$ can be obtained from $\uend(k[x,y])$ as follows: let
$$
\uaut(k[x,y])=\frac{k \langle a,b,c,d,\delta,\delta^{-1} \rangle}{I},
$$
where $I$ is the ideal generated by the relations:
\begin{equation}
\label{eq:relations}
\begin{aligned}
ac-ca & =  0 =  bd-db \\ 
ad-cb & =  \delta  =  da-bc\\
\delta \delta^{-1} & =  1  =  \delta^{-1}\delta,\\ 
a\delta^{-1}d-b\delta^{-1}c & =  1  =  d\delta^{-1}a-c\delta^{-1}b,\\
b\delta^{-1}a-a\delta^{-1}b & =  0  =  c\delta^{-1}d-d\delta^{-1}c
\end{aligned}
\end{equation}
The bialgebra structure is the one above, extended by:
\begin{align}
\Delta(\delta^{\pm 1}) &= \delta^{\pm 1} \otimes \delta^{\pm 1}, \\
\epsilon(\delta^{\pm 1}) &=1.
\end{align}
The antipode is determined by 
\begin{align}
\label{eq:solveforSM}
S(M)
&=
\begin{pmatrix}
\delta^{-1}d & -\delta^{-1}b \\ 
-\delta^{-1}c & \delta^{-1}a
\end{pmatrix} \\
S(\delta^{\pm 1}) &=\delta^{\mp 1}.
\end{align}
	\begin{proposition}
	The Hopf algebra defined above is the universal coacting Hopf algebra of $k[x,y]$.
	\end{proposition}
	\begin{proof}
This follows by implementing the procedure outlined in Remark \ref{rem:Hopf}. We will only sketch it. Since $\delta$ is
grouplike
the symbol $s(\delta)$ satisfies $s(\delta)\delta=\delta s(\delta)=1$ (by \ref{rem:Hopf}(4)) and hence $s(\delta)$
is a twosided inverse of $\delta$ which we denote by $\delta^{-1}$. 

Also by \ref{rem:Hopf}(4) we have
\[
s(M)M=\Id=Ms(M)
\]
It turns out that $\Id=Ms(M)$ can be solved and yields 
\[
\label{eq:solveforSM}
s(M)
=
\begin{pmatrix}
\delta^{-1}d & -\delta^{-1}b \\ 
-\delta^{-1}c & \delta^{-1}a
\end{pmatrix}
\]
Plugging the solution into $s(M)M$
yield the 4 last relations in \eqref{eq:relations}. Having done this it turns out that the relations in Remark \ref{rem:Hopf} imply that the $s^n(M)$ are all expressible in $a,b,c,d,\delta^{-1}$ for $n\ge 2$. Hence we find that $\uaut(k[x,y])$
as an algebra is described by \eqref{eq:relations}. The only thing that remains to be done is to extend $\Delta$ to
$\uaut(k[x,y])$
and define $S$ on it, using the formules Remark in \ref{rem:Hopf}. This finishes the proof.
	\end{proof}
\begin{remark}
        Notice that this Hopf algebra even has a bijective antipode,
        so it also fulfills the universal property of
        Theorem~\ref{hopf} if one demands it to be universal amongst
        Hopf algebras with bijective antipode. 
\end{remark}
Since there is an
        obvious Hopf algebra epimorphism
\begin{align}
\label{quotient}
\uaut(k[x,y]) \twoheadrightarrow \Oscr(\GL_2)
\end{align}
to the coordinate ring of the reductive algebraic group $\GL_2$, we denote $\uaut(k[x,y])$ by $\Oscr_{\tt nc}({{\GL_2}})$, and think of it as the coordinate ring of a noncommutative version of ${{\GL_2}}$. 

\begin{remark}
One can check that $S^2 \neq 1$ and this Hopf algebra is neither braided nor cobraided. 
\end{remark}

To facilitate the computations later on, we introduce a convenient basis for this Hopf algebra.

	\begin{lemma}
	\label{basis}
	The Hopf algebra $\Oscr_{\tt nc}({{\GL_2}})$ has a basis of the form
	$$
	\delta^{x_{1}}w_{1}\delta^{x_{2}}w_{2} \ldots w_{n} \delta^{x_{n}},
	$$
	where $x_{i} \in \z$, $x_{i} \neq 0$ for $i \notin \{1,n\}$, and the $w_{i}$ are non-empty words in the 	symbols $a,b,c,d$ with non-decreasing row index. If $x_{i}=-1$, and $i \notin \{1,n\}$ then the column 	index of the symbol on the left and on the right of $\delta^{x_{i}}=\delta^{-1}$ should be 
	non-decreasing as well.
	\end{lemma}
	\begin{proof}
	This is a routine application of the Bergman diamond lemma using the ordering 
	$
	\delta^{-1}<\delta<a<b<c<d.
	$
	\end{proof}

Just like in~\cite{kriegk-vandenbergh}, one could consider $\Oscr_{\tt nc}({{\GL_2}})$ as a graded algebra in the obvious way, and study the representations of the $i$-th graded piece. Unlike for $\Oscr_{\tt nc}(M_2)$ however, the corresponding degree $i$ subcoalgebras are not finite dimensional, so this is not very useful. In the commutative setting, it is easy to pass between rational and polynomial representations, and one reduces this problem to the polynomial representation theory of $\Oscr(M_2)$. Since $\delta$ is not central in $\Oscr_{\tt nc}({{\GL_2}})$, this does not work in our setting.

\section{Intrinsic standard, costandard and simple comodules}
\label{section:standard}
In this  section, we introduce $\Oscr_{\text{nc}}({{\GL_2}})$-comodules $\Delta_I(\lambda)$, $\nabla_I(\lambda)$ which will 
eventually be shown to be the standard and costandard modules for a suitable
quasi-hereditary structure on $\Oscr_{\text{nc}}({{\GL_2}})$.
\subsection{Some canonical representations and their weights}
Put 
	\begin{equation}
	\Oscr(T):=\Oscr_{\tt nc}({{\GL_2}})/(b,c)=k[a^{\pm 1},d^{\pm 1}].
	\end{equation}
We see that $\Oscr(T)$ is the actual 
coordinate ring of a commutative two-dimensional torus. We will
identify the character group $X(T)$ (``weights'') of $T$ with the Laurent monomials in $a$, $d$.
%
We
give the weights the lexicographical ordering for $a<d$, i.e. $a^i
d^j < a^{i'}d^{j'}$ iff $j < j'$ or $j=j'$ and $i < i'$. Define two
involutions $(-)^*$ and $\sigma$ on the weights by
\begin{equation}
\begin{aligned}
(a^x d^y)^* &=a^{-y}d^{-x}, \\
\sigma(a^x d^y) &=a^y d^x.
\end{aligned}
\end{equation}
These involutions are incarnations of the action of the non-trivial Weyl group element of ${{\GL_2}}$. We now define the partially ordered set indexing the simples of $\Oscr_{\tt nc}({{\GL_2}})$.

	\begin{definition}
	\label{definition:lambda}
	The set $\Lambda$ consists of all formal expressions of the form
	\begin{equation}
	\lambda := \delta^{x_1} d^{y_1} \cdots \delta^{x_n} d^{y_n},
	\end{equation}
	where $x_i \in \z$ and $y_i \in \n$. 
We define a ``weight function'' 
on $\Lambda$ as
follows:
	\begin{equation}
	{\tt wt}:\Lambda \to X(T):\lambda \mapsto a^{\sum x_i}d^{\sum x_i+y_i}.
	\end{equation}
	\end{definition}
In particular ${\tt wt}(\delta)=ad$, ${\tt wt}(d)=d$. 

Note that ${\tt wt}$ is not surjective. Its image consists those weights
$a^x d^y$ for which $y\ge x$. We will put $X(T)^+=\im {\tt wt}$. The
elements of $X(T)^+$ will be called dominant weights. 

        Elements of $\Lambda$ are ordered according to the ordering on
        $X(T)$. I.e.  $\mu<_2\lambda$ if and only if
        ${\tt wt}(\mu)<{\tt wt}(\lambda)$. In particular elements of
        $\Lambda$ with the same weight are considered incomparable,
        unless they are equal. The ordering is denoted by $<_2$ since
        later we will introduce a finer one denoted by $<_1$. 

 The
        map $(-)^*$ is defined on $\Lambda$ by demanding that
\begin{equation}
\begin{aligned}
d^* &=d\delta^{-1} \\
\delta^* &=\delta^{-1} \\
(\lambda \mu)^* &=\mu^* \lambda^*
\end{aligned}
\end{equation} 
With this definition, we
have that ${\tt wt}(\lambda^*)={\tt wt}(\lambda)^*$. Notice however that $(-)^\ast:\Lambda\rightarrow \Lambda$ is not an involution. 

        The weights of a $\Oscr_{\tt nc}({{\GL_2}})$-representation $X$ are
        defined in the standard way, i.e.~$X$ may be considered as an
        $\Oscr(T)$-comodule via the composition
\[
X\rightarrow \Oscr_{\tt nc}({{\GL_2}})\otimes X\rightarrow \Oscr(T)\otimes X
\]
so one can decompose $X$ into one-dimensional,
        simple torus representations $k_t$, for~$t$ a monomial in
        $k[a^{\pm 1},d^{\pm 1}]$, and
\begin{equation}
k_t \xrightarrow{\delta_2} \Oscr(T) \ot k_t:1 \mapsto t \ot 1.
\end{equation}
When no confusion can arise, we will abbreviate $k_t$ by $t$. 

Let $R=k r$ and $R^{-1}=k r^{-1}$ be the one-dimensional
comodules defined by
\begin{equation}
r^{\pm 1} \mapsto \delta^{\pm 1} \ot r^{\pm 1},
\end{equation}
and let $V=k e_1 +k e_2$ be the two-dimensional comodule defined by
\begin{equation}
\begin{aligned}
\begin{pmatrix} e_1 \\ e_2 \end{pmatrix}
\mapsto 
\begin{pmatrix} a & b \\ c & d \end{pmatrix}
\otimes
\begin{pmatrix} e_1 \\ e_2 \end{pmatrix}
\end{aligned}
\end{equation}
\begin{definition}
\label{definition:nabla}
For $\lambda \in \Lambda$ as in Definition~\ref{definition:lambda}, put
$
M(\ll)=R^{\ot x_1} \ot V^{\ot y_1} \ot \cdots \ot R^{\ot x_n} \ot V^{\ot y_n}
$
and let $\nabla_I(\lambda)$ be the subcomodule of the regular comodule $\oo$ spanned by vectors
\begin{equation}
\delta^{x_1}b^{y'_1}d^{y''_1} \cdots \delta^{x_n} b^{y'_n}d^{y''_n},
\end{equation}
where $y'_i+y''_i=y_i$.
\end{definition}

From now on, we will often drop tensor signs to compactify the notation. 
Recall that the right dual  of an object $X$ in a monoidal category  is 
 a triple $(X^\ast,\eval_X,\coeval_X)$ consisting of an object $X^\ast$  and morphisms
\begin{equation}
  \eval_X:X^* \otimes X \to 1 \textnormal{ and } \coeval_X:1 \to X \otimes X^*,
\end{equation}
such that the compositions 
\begin{equation}
  X \xrightarrow{\coeval_X \otimes 1} X \otimes X^* \otimes X \xrightarrow{1 \otimes 
    \eval_X} X,
\end{equation}
and
\begin{equation}
  X^* \xrightarrow{1 \otimes \coeval_X} X^* \otimes X \otimes X^* \xrightarrow{\eval
    _X \otimes 1} X^*
\end{equation}
are the identity morphisms. The left dual ${}^\ast\! X$ is defined similarly. Duals are unique
up to unique isomorphism. Usually we will just write ${}^\ast\! X$, $X^\ast$, leaving
the evaluation and coevaluation morphisms implicit.

	\begin{lemma}
	With the above conventions, one has that $V^* \cong VR^{-1}, {}^*V \cong R^{-1}V$ and 
	$R^* \cong {}^*R \cong R^{-1}$.
	\end{lemma}
	\begin{proof}
We need to specify the evaluation and 
	coevaluation morphisms.  For $V$, it is easy to check that these morphisms are given by 
	\begin{equation}
	\begin{aligned}
	\eval_V &: VR^{-1}V \to k : 
	\begin{pmatrix} 
	e_1 r^{-1} e_1 & e_1 r^{-1} e_2 \\ 
	e_2 r^{-1} e_1 & e_2 r^{-1} e_2 
	\end{pmatrix}
	\mapsto
	\begin{pmatrix} 
	0 & -1 \\ 
	1 & 0 
	\end{pmatrix},
	\\
	\coeval_V &: k \to VVR^{-1} :
	1 \mapsto (e_1e_2 - e_2e_1)r^{-1}.
	\end{aligned}
	\end{equation}
	For left duals the proof is similar and for $R$ it is even easier. 
	\end{proof}
One  finds in particular 
\begin{equation}
\label{eq:Miso}
M(\lambda)^* \cong M(\lambda^*)\,.
\end{equation}
Below we will write $\Delta_I(\lambda^*)$ for $\nabla_I(\lambda)^\ast$.

The rather cumbersome formulation of the following lemma is due to the fact that neither
$(-)^\ast$ nor $\sigma(-)$ is compatible with the ordering $<_2$.
\begin{lemma} 
\label{lem:bijection} Both $\nabla_I(\lambda)$ and $M(\lambda)$ possess highest and lowest weights
$\mathtt{wt}(\lambda)$, $\sigma(\mathtt{wt}(\lambda))$ as $\Oscr(T)$-represenations (with the ordering $<$ introduced above),
each occurring with multiplicity one. The same holds for $\nabla_I(\lambda)^\ast$ and $M(\lambda)^\ast$
where the highest weights and lowest weights are respectively $\mathtt{wt}(\lambda)^\ast$
and $\sigma(\mathtt{wt}(\lambda))^\ast$.

Moreover the obvious epimorphism (of comodules)
\begin{equation}
M(\lambda) \twoheadrightarrow \nabla_I(\lambda)
\end{equation}
is  a bijection on highest and lowest weight vectors. The same holds for  the dual monomorphism,
\begin{equation}
\nabla_I(\lambda)^\ast \hookrightarrow  M(\lambda)^\ast
\end{equation}
\end{lemma}
\begin{proof}
Let $\lambda$ be as in Definition \ref{definition:lambda}
and let us consider $\nabla_I(\lambda)$. The weight of a basis vector
\begin{align}
\label{bbb}
\delta^{x_1}b^{y'_1}d^{y''_1} \cdots \delta^{x_n} b^{y'_n}d^{y''_n},
\end{align}
where $y'_{i}+y''_{i}=y_{i}$ is
\[
	a^{\sum (x_{i}+y'_i)} d^{\sum (x_{i}+y''_i)},
\] 
which is maximal if $y_i=y''_i$ and minimal if $y_i=y'_i$. In both cases we see that the weights
are as indicated.

The weights of the basis vectors of  $\nabla_I(\lambda)^\ast$ are 
\[
(a^{\sum (x_{i}+y'_i)} d^{\sum (x_{i}+y''_i)})^\ast=a^{-\sum (x_{i}+y''_i)} d^{-\sum (x_{i}+y'_i)},
\]
which is again maximal if $y_i=y''_i$ and minimal if $y_i=y'_i$. The weights are once
again as indicated. 

The arguments for $M(\lambda)$ are the same and the other claims of the lemma
are obvious.
\end{proof}
\subsection{Filtered coalgebras}
As a preparation for the sequel we remind the reader of some basic properties of filtered coalgebras.
	\begin{definition}
          A filtered coalgebra $C$ is a coalgebra $C$ equipped with a
          filtration $C=\cup_{n \geq 0} C_{n}$, where $(C_{n})_{n}$ is
          an ascending chain of subspaces, satisfying
	\begin{align}
	\label{filter}
	\Delta(C_{n}) \subset \sum_{m \geq 0} C_{m} \ot C_{n-m}.
	\end{align}
	\end{definition}

	\begin{lemma}
	\label{filtertje}
	For a filtered coalgebra $C$, and a $C$-comodule $V$, there exists a non-trivial subspace $V_{0}$ 	that is a $C_{0}$-comodule.
	\end{lemma}
	\begin{proof}
	The coaction $\delta:V \to C \ot V$ of any element $v \in V$ can be decomposed in such a way as to 
	respect the filtration:
	$$
	\delta(v)=\sum_{n,i} c_{n,i} \ot v_{n,i},
	$$
	if we take the $(c_{n,i})_{n,i}$ to be preimages of the bases $(\bar{c}_{n})_{i}$ of $C_{n}/C_{n-1}$ for 	the natural quotient maps $C_{n} \to C_{n}/C_{n-1}$. Now define $V_{0}$ to be the span of $(v_{N,i})_	{i}$, with $N$ maximal among the $n$ for which there exists a non-zero $v_{n,i}$ in $\delta(v)$. Since 	$V$ is a comodule, we have
	$$
	\sum_{n,i} \Delta(c_{n,i}) \ot v_{n,i} = \sum_{n,i} c_{n,i} \ot \delta(v_{n,i}) \in C \ot C \ot V.
	$$
	Reducing to $C_{N}/C_{N-1} \ot C \ot V$, and noticing that because $C$ is filtered, the 		
	comultiplication descends to a map $\bar{\Delta}:C_{N}/C_{N-1} \to C_{N}/C_{N-1} \ot C_{0}$, the 
	above equality provides us with the inclusion
	$$
	\sum_{i} \bar{c}_{N,i} \ot \delta(v_{N,i}) \subset C_{N}/C_{N-1} \ot C_{0} \ot V_{0}.
	$$
	Since the $(\bar{c}_{N,i})_{i}$ form a basis, we have that $\delta(V_{0}) \subset C_{0} \ot V_{0}$.
	\end{proof}
	
	\begin{corollary}
	\label{corollary:filtered}
	All group like elements $g$ of a filtered coalgebra $C$ lie in $C_{0}$.
	\end{corollary}
	\begin{proof}
	If $g \in C_{N} \backslash C_{N-1}$, then $0 \neq \bar{\Delta}(g)=g \ot g \in C_{N}/C_{N-1} \ot C_{N}/C_{N-1}$. This contradicts \eqref{filter}, unless $N=0$. 
\end{proof}

	\begin{definition}
	Given a group like element $g$ in $C$, a non-zero vector $v \in V$ is called a semi-invariant of 	
	weight $g$ if 
	$$
	\delta(v)=g \ot v.
	$$
	\end{definition}

	\begin{lemma}
	\label{lemma:grouplike}
	If $V$ is a subrepresentation of the regular representation $C$, then the semi-invariants are
	scalar multiples of the grouplike elements.
	\end{lemma}
	\begin{proof}
	In this case $\delta=\Delta$, and by applying $1 \ot \epsilon$ to $\delta(v)=g \otimes v$ and using 	
	counitality on the left hand side, we find that $v=g\epsilon(v)$.
	\end{proof}	

\subsection{Borel coalgebras}
One has the following analogues of the (coordinate rings of) Borel subgroups:
\begin{equation}
\begin{aligned}
\Oscr_{\tt nc}(B)&=\Oscr_{\tt nc}({{\GL_2}})/(b) \cong k \langle c,d^{\pm 1} \rangle [a^{\pm 1}], \\
\Oscr_{\tt nc}(B^+)&=\Oscr_{\tt nc}({{\GL_2}})/(c) \cong k \langle a^{\pm 1},b \rangle [d^{\pm 1}]
\end{aligned}
\end{equation}
which are non-commutative quotient Hopf algebras of $\Oscr_{\tt nc}({{\GL_2}})$ with quotient maps~$\pi$ (respectively~$\pi^+$). Note that there is a commutative diagram
$$
\begin{tikzcd} 
\Oscr_{\tt nc}({{\GL_2}}) \arrow{r}{\psi} \arrow{d}{\pi} & \Oscr_{\tt nc}({{\GL_2}}) \arrow{d}{\pi^+} \\ 
\Oscr_{\tt nc}(B) \arrow{r}{\psi}[swap]{\cong} & \Oscr_{\tt nc}(B^+) 
\end{tikzcd}
$$
where~$\psi$ denotes the Hopf algebra automorphism
\begin{equation}
\label{psi}
	\psi:\Oscr_{\tt nc}({{\GL_2}}) \to \Oscr_{\tt nc}({{\GL_2}}):
	\begin{pmatrix}
	a & b \\ c & d
	\end{pmatrix}
	\mapsto
	\begin{pmatrix}
	d & c \\ b & a
	\end{pmatrix}.
	\end{equation}

	\begin{lemma}
	$\Oscr_{\tt nc}(B^+)$ is a pointed, filtered coalgebra. The filtration is defined as follows: if 
	$\Oscr_{\tt nc}(B^+)_{n}' \subset \Oscr_{\tt nc}(B^+)$ is the span of the monomials in $a, b$ and $d$ that contain 
	$n$ copies of $b$, then 
	$$
	\Oscr_{\tt nc}(B^+)_{n}=\oplus_{m \leq n} \Oscr_{\tt nc}(B^+)_{m}'.
	$$
	\end{lemma}
	\begin{proof}
	First notice that $\Oscr_{\tt nc}(B^+)$ is generated by grouplike and skew-primitive elements, so it is 
	pointed. Also, $\Oscr_{\tt nc}(B^+)_{0}'$ is spanned by the products $a^{s}d^{t}$. We have
	$$
	\Delta(b)= a \ot b + b \ot d,
	$$
        Since $\Delta$ is homogeneous for the grading $|a|=|d|=0$,
        $|b|=1$ it follows that $\Oscr_{\tt nc}(B^+)=\oplus_{n \geq
          0}\Oscr_{\tt nc}(B^+)_{n}'$ is an $\mathbb{N}$-graded
        coalgebra, so $\Oscr_{\tt nc}(B^+)$ becomes filtered by
        setting
	$$
	\Oscr_{\tt nc}(B^+)_{n} = \oplus_{m \leq n} \Oscr_{\tt nc}(B^+)_{m}'.
	$$
	\end{proof}
	
	\begin{corollary}
	The group like elements in $\Oscr_{\tt nc}(B^+)$ are all of the form $a^id^j$ for $i,j\in \ZZ$.
	\end{corollary}
	\begin{proof}
	By Corollary~\ref{corollary:filtered}  all group like elements in $\Oscr_{\tt nc}(B^+)$ are contained in
	$\Oscr_{\tt nc}(B^+)_{0}=k[a^{\pm 1},d^{\pm 1}]$, which is exactly $\Oscr(T)$. It now suffices
to note that the grouplike elements in $\Oscr(T)$ have the indicated form.
	\end{proof}

The following lemma is a noncommutative version of the Lie-Kolchin theorem.

	\begin{lemma}
	\label{lemma:semi}
	Every $\Oscr_{\tt nc}(B^+)$-representation contains a semi-invariant.
	\end{lemma}
	\begin{proof}
          From Lemma~\ref{filtertje} we know that every $\Oscr_{\tt
            nc}(B^+)$-comodule $V$ contains a $\Oscr_ {\tt
            nc}(B^+)_{0}$-comodule $V_{0}$. Since $\Oscr_{\tt
            nc}(B^+)_{0}=\Oscr(T)$ is the coordinate ring of a torus,
          $V_{0}$ is spanned by semi-invariants.
	\end{proof}
	\begin{proposition}
	\label{nab}
	Every subrepresentation of $\nabla_I(\ll)$ contains $\ll$, viewed as highest weight vector in 
	$\nabla_I(\ll)$.
	\end{proposition}
	\begin{proof}
	The composition
	$$
	\nabla_I(\ll) \hookrightarrow \Oscr_{\tt nc}({{\GL_2}}) \xrightarrow{\pi} \Oscr_{\tt nc}(B^+)
	$$
	sends a basis vector, say
	\begin{align}
	\label{bbb}
	\delta^{x_1}b^{y'_1}d^{y''_1} \cdots \delta^{x_n} b^{y'_n}d^{y''_n},
	\end{align}
	where $y'_{i}+y''_{i}=y_{i}$, to
	\begin{align}
	\label{bb}
	a^{x_1}b^{y'_1} \cdots a^{x_n}b^{y_n'}d^{\sum (x_i + y''_i)}
	\end{align}
	Hence the map is injective, and we can view $\nabla_I(\ll)$ as subrepresentation of $\Oscr_{\tt nc}(B^+)$. 
	From Lemma~\ref{lemma:grouplike}, it follows that $\nabla_I(\ll)$ contains at most one semi-invariant of 
	a 
	fixed weight $a^{u}d^{l}$ (up to scalar multiple), and in that case this element is exactly $a^{u}d^{l}$, if 
	it sits in $\nabla_I(\ll) \hookrightarrow \Oscr_{\tt nc}(B^+)$. From \eqref{bb}, we see immediately 
	that only one weight, and thus only one semi-invariant can and does occur, namely
	$$
	a^{\sum x_{i}} d^{\sum (x_{i}+y_i)},
	$$ 
	which corresponds to $t={\tt wt}(\ll)$. Converting to $\nabla_I(\ll)$ as a subcomodule of $\Oscr_{\tt nc}({{\GL_2}})$ 
	again, we see that $\ll$ is the unique $\Oscr_{\tt nc}(B^+)$-semi-invariant in $\nabla_I(\ll)$. Suppose now that $V$ is 
	a subrepresentation of $\nabla_I(\ll)$; by Lemma~\ref{lemma:semi}, we know that $V$ contains a 
	$\Oscr_{\tt nc}(B^+)$-semi-invariant $v$ of weight $t$. Then by the previous considerations, $v$ must be a multiple of $\ll$, proving the proposition.
	\end{proof}

\subsection{Applications}
	\begin{proposition}
	\label{llll}
	The comodules $\nabla_I(\lambda)$ are Schurian and 
	the subrepresentation $L(\ll)$ of $\nabla_I(\ll)$ cogenerated by $\ll$ is simple. 
	\end{proposition}
	\begin{proof}
By Proposition~\ref{nab} any subcomodule of $L(\lambda)$ must contain $\lambda$.
So it must be equal to $L(\lambda)$.

The proof that $\nabla_I(\lambda)$
is Schurian is similar.
          A non-zero endomorphism $f$ of $\nabla_I(\lambda)$ gives
          rise to an endomorphism of $\Oscr_{\tt nc}(B^+)$-comodules,
          which will also be denoted~$f$. If the kernel of $f$ is
          non-zero then it contains $\lambda$ by
          Proposition~\ref{nab}. But then the image of $f$ cannot
          contain $\lambda$ which contradicts Proposition~\ref{nab} unless the image
is zero. So $f$ must either be zero or an automorphism.

Since $\lambda$ is the unique 
	semi-invariant of weight ${\tt wt}(\lambda)$ in $\nabla_I(\lambda)$ up to scalar multiplication, one has 
	$f(\lambda)=c \lambda$. Since $f-c$ is not an automorphism we deduce $f=c$, finishing the proof.
	\end{proof}

\begin{proposition}
\label{prop:Ll}
We have $L(\lambda)^\ast\cong L(\lambda^\ast)$ and moreover both are equal to the image 
of the composition
\begin{equation}
\label{eq:comp}
\nabla_I(\lambda)^* \hookrightarrow M(\lambda)^* \xrightarrow[\eqref{eq:Miso}]{\cong} M(\lambda^*) \twoheadrightarrow \nabla_I(\lambda^*)
\end{equation}
\end{proposition}
\begin{proof}
\eqref{eq:comp} is  a bijection on the highest and lowest weight vectors by Lemma \ref{lem:bijection}. Taking kernels and cokernels of the composed morphism, there is an exact sequence
\begin{equation}
\label{exactsequence}
0 \to K \to \nabla_I(\lambda)^* \to \nabla_I(\lambda^*) \to C \to 0.
\end{equation}
which may be completed to a diagram:
\[
\begin{tikzcd}
	& & && L(\lambda^*) \dlar[dashed,hookrightarrow]\drar[dashed]{0} \dar[hookrightarrow] & &\\
	0 \rar & K \drar[dashed][swap]{0} \rar & \nabla_I(\lambda)^* \rar[twoheadrightarrow] \dar[twoheadrightarrow] 
&Z\dlar[dashed,twoheadrightarrow]
\rar[hookrightarrow]& \nabla_I(\lambda^*) \rar & C \rar & 0 \\
	& & L(\lambda)^* && & & 
\end{tikzcd}
\]
where the outher dashed arrows are $0$ because $K$ and $C$ have weights
strictly between ${\tt wt}(\lambda^*)$ and $\sigma ({\tt wt}
(\lambda^*))$. If the resulting composition 
\begin{equation}
\label{split}
L(\lambda^\ast)\hookrightarrow Z\twoheadrightarrow L(\lambda)^\ast
\end{equation}
is zero then the weight $\mathsf{wt}(\lambda^\ast)$ occurs twice in $Z$ and hence in $\nabla_I(\lambda^\ast)$
which is impossible by Lemma \ref{lem:bijection}. Hence we conclude that $L(\lambda^\ast)=L(\lambda)^\ast$
and the inclusion $L(\lambda^\ast)\hookrightarrow Z$ is split. If it not an isomorphism then
$ \nabla_I(\lambda^*)$ contains a decomposable submodule which is impossible by Proposition \ref{nab}.
\end{proof}
\subsection{A canonical filtration on $\Oscr_{\tt nc}({{\GL_2}})$}
\label{sec:filtration}
 Write~$\Oscr_{\tt nc}({{\GL_2}})$ as an ascending union of finite dimensional subcoalgebras:
\begin{equation}
\label{subcoalgebra}
\Oscr_{\tt nc}({{\GL_2}})= \cup_{n \geq 0} \Oscr_n,
\end{equation}
where $\Oscr_n$ is the subcoalgebra consisting of all elements that
can be written as linear combinations of words of length $\leq n$ in
the generators $a,b,c,d$ and $\delta$, $\delta^{-1}$ (thus each generator has
length $1$). 
Let $I_n$ be the set of words in $c,d,\delta,\delta^{-1}$ of
length $n$ not containing $d\delta^{-1}c$, $\delta\delta^{-1}$, $\delta^{-1}\delta$ and let $t:I_n\rightarrow \Lambda$ be the map which replaces $c$ by $d$.  
\begin{lemma} 
\label{lemma:filtration}
One has as left $\Oscr_{\tt nc}({{\GL_2}})$-comodules
\[
\Oscr_n/\Oscr_{n-1}\cong \bigoplus_{\gamma\in I_n} \nabla_I(t(\gamma))
\]
\end{lemma}
\begin{proof}
Let $J_n$ be the set of words in $a,b,c,d,\delta,\delta^{-1}$ of length
$n$ as introduced in Lemma \ref{basis}. It is clear that $J_n$
yields a basis for $\Oscr_n/\Oscr_{n-1}$.

Let $s:J_n\rightarrow I_n$ be the map which replaces $a$ by $c$ and $b$ by $d$. For 
$\gamma\in I_n$ define
\[
\widehat{\nabla}_I(\gamma)=\bigoplus_{w\in J_n,s(w)=\gamma} k\bar{w}\subset \Oscr_n/\Oscr_{n-1}
\]
It easy to see that $\widehat{\nabla}_I(\gamma)$ is a subcomodule
of $\Oscr_n/\Oscr_{n-1}$ and furthermore
\[
\widehat{\nabla}_I(\gamma)\cong \nabla_I(t(\gamma))
\]
This finishes the proof.
\end{proof}
\subsection{Induced representations}
\label{sec:induced}
Now we  look at the induced representations  
	\begin{equation}
	\induced_B^{{\GL_2}}(t):=\Oscr_{\tt nc}({{\GL_2}}) \boxtimes^{\Oscr_{\tt nc}(B)} k_t,
	\end{equation}
where $\boxtimes$ denotes the cotensorproduct. Also, $\Oscr_{\tt nc}({{\GL_2}})$ is regarded as a right~$\Oscr_{\tt nc}(B)$-comodule by comultiplying and composing with the quotient map:
$$
\Oscr_{\tt nc}({{\GL_2}}) \xrightarrow{\Delta} \Oscr_{\tt nc}({{\GL_2}}) \ot \Oscr_{\tt nc}({{\GL_2}}) \xrightarrow{1 \otimes \pi}  
\Oscr_{\tt nc}({{\GL_2}}) \ot \Oscr_{\tt nc}(B).
$$
This composition is denoted $\delta_1$. The $\Oscr(T)$-comodule $k_t$ can be regarded as left $\Oscr_{\tt nc}(B)$-comodule in the obvious way, with corresponding comodule map $\delta_2$. Remember that the cotensor product is defined by 
$$
\Oscr_{\tt nc}({{\GL_2}}) \boxtimes^{\Oscr_{\tt nc}(B)} k_t:=\tn{Ker}(\Oscr_{\tt nc}({{\GL_2}}) \ot k_t \xrightarrow{\phi} \Oscr_{\tt nc}({{\GL_2}}) \ot \Oscr_{\tt nc}(B) \ot k_t),
$$
where $\phi=\delta_1 \ot 1 - 1 \ot \delta_2$.
	\begin{theorem}
	\label{induced}
	The induced representations decompose as a direct sum of $\nabla_I$'s, i.e.
	$$
	\induced_B^{{\GL_2}}(t) = \bigoplus_{\substack{\lambda \in \Lambda \\ wt(\lambda)=t}} \nabla_I(\lambda).
	$$
	In particular, for $t \notin X(T)^+, \induced_B^{\GL_2}(t)=0$.
	\end{theorem}
	\begin{proof}
	First note that one can compute $\induced_B^{{\GL_2}}(t)$ from the right~$\Oscr_{\tt nc}(B)$-semi-invariants of 
	weight $t$ for $(\Oscr_{\tt nc}({{\GL_2}}),\delta_1)$ since $f \in \induced_B^{{\GL_2}}(t)$ iff
	$
	\delta_1(f) \otimes 1=f \otimes t \otimes 1.
	$
	If $f$ is a right~$\Oscr_{\tt nc}(B)$-semi-invariant of weight $t$, one finds (using 
	Sweedler notation)
	\begin{equation}
	\begin{aligned}
	& f_{(1)} \ot f_{(2)} = f \ot t \\
	\Rightarrow  &  (S \ot S)(f_{(1)} \ot f_{(2)}) = (S \ot S)(f \ot t) \\
	\Rightarrow &  \Delta^{op} \circ S(f) = S(f) \ot \sigma(t^*) \\
	\Rightarrow & S(f)_{(2)} \ot S(f)_{(1)} = S(f) \ot \sigma(t^*)\\
	\Rightarrow & S(f)_{(1)} \ot S(f)_{(2)} =  \sigma(t^*) \ot S(f) 
	\end{aligned}
	\end{equation}
	where we used $(S \ot S) \circ \Delta=\Delta^{op} \circ
        S$. Note that we suppressed the quotient map~$\pi$. One finds
        that $S(f)$ is a semi-invariant of weight $\sigma(t^*)$ for
        the left $\Oscr_{\tt nc}(B)$-comodule $(\Oscr_{\tt nc}({{\GL_2}}),
        \Delta)$. A similar easy computation shows that~$h$ is a
        left~$\Oscr_{\tt nc}(B)$-semi-invariant of weight~$t'$ if and
        only if~$\psi(h)$ (see~\eqref{psi}) is a left~$\Oscr_{\tt
          nc}(B^+)$-semi-invariant of weight~$\sigma(t')$. All in all
        one obtains that if~$f$ is a right~$\Oscr_{\tt
          nc}(B)$-semi-invariant of weight~$t$, then~$(\psi \circ
        S)(f)$ is a left~$\Oscr_{\tt nc}(B^+)$-semi-invariant of
        weight~$t^*$. So we might as well compute the left~$\Oscr_{\tt
          nc}(B^+)$-semi-invariants.

We may refine the $(\Oscr_n)_n$ constructed in Lemma \ref{lemma:filtration}
        to an exhaustive ascending filtration $(F_n)_n$ with $F_0=0$ on
        $\Oscr_{\tt nc}({{\GL_2}})$ such that $F_{n+1}/F_n$ is isomorphic to
        $\nabla_I(\lambda)$ for suitable $\lambda$. 

Moreover, we know
        from Proposition~\ref{nab} that each $\nabla_I(\ll)$ contains a
        unique~$\Oscr_{\tt nc}(B^+)$-semi-invariant of weight $t={\tt
          wt}(\lambda)$. Using this in combination with~\eqref{bbb}
        and~\eqref{bb}, we know that the left~$\Oscr_{\tt
          nc}(B^+)$-semi-invariants in the associated graded of
        $\Oscr_{\tt nc}({{\GL_2}})$ consist of polynomials in $c,d,\delta$ and
        $\delta^{-1}$.

	If $f$ is now any left~$\Oscr_{\tt nc}(B^+)$-semi-invariant in $\Oscr_{\tt nc}({{\GL_2}})$, and $l$ is minimal such that $f \in F_l$, then 	$0 \neq f$ in 
	$F_l/F_{l-1}$ and we have a well-defined map
	$$
	F_l/F_{l-1} \xrightarrow{\delta_1} \Oscr_{\tt nc}(B^+) \ot F_l/F_{l-1}:f \mapsto t \ot f,
	$$
	so we know $\bar{f} \in F_l/F_{l-1}$ is some polynomial $g$ of length $l$ in $c,d,\delta$ and $\delta^{-1}$ and one checks that the latter are $\Oscr_{\tt
          nc}(B^+)$-semi-invariants in $\Oscr_{\tt nc}({{\GL_2}})$. So $f-g$ is a semi-invariant in $F_{l-1}$. 
By induction (using $F_0=0$)	
	 any left semi-invariant in 
	$\Oscr_{\tt nc}({{\GL_2}})$ is a polynomial in $c,d,\delta$ and $\delta^{-1}$.
Applying $(\psi \circ S)^{-1}$ to 
	these 
	polynomials, we get polynomials in $b,d,\delta^{-1}$ and $\delta$. In terms of the explicit basis of 
	$\Oscr_{\tt nc}({{\GL_2}})$, this says that if $t \in \tn{Im}({\tt wt})$, then $\induced_B^{{\GL_2}}(t)$ consists of 
	elements of the form (we no longer write the $\ot v$ all the time)
	$$
	\delta^{n_0}b^{p_1}d^{q_1} \delta^{n_1} \cdots b^{p_k}d^{q_k}\delta^{n_k},
	$$
	where the $n_{r} \in \z, p_{r},q_{r} \in \n$
which have weight $t$. 
This gives the indicated direct sum decomposition. Note that this also shows that for $t \notin X(T)^+, \induced_B^{\GL_2}(t)=0$.
	\end{proof}

\begin{remark}
Notice that this is different from the reductive algebraic group setting, where the induced representations are indecomposable, see Jantzen~\cite[II.2.8]{jantzen2}.
\end{remark}

\section{A quasi-hereditary filtration of~$\Oscr_{\tt nc}({{\GL_2}})$}

\subsection{A filtration by quasi-hereditary subcoalgebras}
Recall the ascending filtration on $\Oscr_{\tt nc}({{\GL_2}})=\cup_{n \geq 0} \Oscr_n$ introduced in \S\ref{sec:filtration}.
In this section we show the following.
\begin{theorem}
\label{th:qh}
  For every~$n$, the finite dimensional coalgebra~$\Oscr_n$ is
  quasi-hereditary with respect to the poset~$(\Lambda_n,\leq_2)$,
  which is the restriction of the poset~$(\Lambda,\leq_2)$ to words of length $\le n$ from
  Definition~\ref{definition:lambda}. Moreover the (co)standard
  comodules are given by $\Delta_I(\lambda)$, $\nabla_I(\lambda)$ for
  $\lambda\in\Lambda_n$ as defined in
  Definition~\ref{definition:nabla}.
\end{theorem}

Since we will be working with a different ordering later on, we will temporarily denote the costandard comodules corresponding to~$\Oscr_n$ and~$(\Lambda_n,\leq_2)$ by~$\nabla_2$. 

We will first show that for~$\lambda \in \Lambda_n$,
one has~$\nabla_2(\lambda)=\nabla_I(\lambda)$. For this we use
Lemma~\ref{lemma:donkincostandard}. We will then use Proposition \ref{prop:convenient}
to show that $\Oscr_n$ is quasi-hereditary. 
\begin{remark}
Notice that for every $\lambda\in\Lambda$, $\pi(\lambda)$ is infinite for the ordering $\leq_2$ 
so the infinite dimensional coalgebra~$\Oscr_{\tt nc}({{\GL_2}})$ is not quasi-hereditary for $\leq_2$ (see Definition \ref{strong}). 
\end{remark}
Using Frobenius reciprocity for coalgebras (see~\cite{jantzen2} for example),
we have
\begin{equation}
\label{equation:adjunction}
\Hom^{\Oscr_{\tt nc}({{\GL_2}})}(-,\induced_B^{{\GL_2}}(t)) \cong \Hom^{\Oscr_{\tt nc}(B)}(-,t).
\end{equation}
For an inclusion of coalgebras $C\subset D$ and a $C$-comodule $V$ we say that $V$ is defined over $D$
  if its structure map $V \to C \otimes V$ has image in $D \otimes V$.
To prove quasi-hereditarity, we need the following lemma.
\begin{lemma}
\label{lemma:tweights}
An~$\Oscr_{\tt nc}(B)$-representation~$V$ with all weights~$=t$, for some fixed~$t \in \Oscr(T)$, is defined over~$\Oscr(T)$.
\end{lemma}
\begin{proof}
We already proved that for every~$\Oscr_{{\tt nc}}(B)$-representation, there exists a one-dimensional subrepresentation. Denote this representation by~$kv$. Since all weights are equal to~$t$, and the grouplike elements of~$\Oscr_{{\tt nc}}(B)$ are 
in $\Oscr(T)$, this representation is in fact
\begin{equation}
kv \to \Oscr_{{\tt nc}}(B) \otimes kv: v \mapsto t \otimes v.
\end{equation}
It remains to show that any extension between two one-dimensional representations of weight $t$ is split. 
Such extension is a two-dimensional representation with basis $v,w$ such that
\begin{align*}
\delta(v)&=t\otimes v\\
\delta(w)&=t\otimes w + u\otimes v
\end{align*}
such that $u\in \Oscr_{{\tt nc}}(B)$ is in the ideal generated by $c$. From coassociativity one obtains
\begin{equation}
\label{eq:coassoc}
\Delta(u)=t\otimes u+u\otimes t
\end{equation}
Let $\deg_c$ be the grading on $\Oscr_{{\tt nc}}(B)$ by $c$-degree and
define a second grading on  $\Oscr_{{\tt nc}}(B)$ by $|a|=-1$, $|c|=0$, $|d|=1$. Then one checks for 
$h$ homogeous in $c$:
\[
|h_{(1)}|-|h_{(2)}|=\deg_c h
\]
It follows that if $p\otimes q$ is a term on the righthand side of \eqref{eq:coassoc} then $|p|-|q|>0$. But this 
is clearly a contradiction since the righthand side of \eqref{eq:coassoc} is preserved under $p\otimes q\mapsto q\otimes p$.
So we obtain $u=0$.
\end{proof}

\begin{proposition}
\label{prop1}
The coinduced comodule $\induced_B^{{\GL_2}}(t)$ is injective in the full subcategory of $\Oscr_{\tt nc}({{\GL_2}})$-representations of weights $\leq_2 t$.
\end{proposition}
\begin{proof}
This follow from the isomorphism~\eqref{equation:adjunction} and the fact that on this subcategory, we have that
$$
\Hom^{\Oscr_{\tt nc}(B)}(-,t) \cong \Hom^{\Oscr(T)}(-,t).
$$
To see this, first note that for representations $M'$ of weights $< _2 t$ we have that
$$
\Hom^{\Oscr(T)}(M',t)=0
$$
and thus also $\Hom^{\Oscr_{\tt nc}(B)}(M',t)=0$ since $\Hom^{\Oscr_{\tt nc}(B)}(M',t) \subset \Hom^{\Oscr(T)}(M',t)$. For a representation $M''$ that has all weights equal to $t$, by Lemma~\ref{lemma:tweights} there is an equality 
$$
\Hom^{\Oscr_{\tt nc}(B)}(M'',t) \cong \Hom^{\Oscr(T)}(M'',t).
$$
Given any $\Oscr_{\tt nc}({{\GL_2}})$-representation $M$ of weights $\leq_2 t$, we now set $M'$ to be the subspace of $M$ generated by all weight vectors of weight $<_2 t$. This is obviously a $\Oscr_{\tt nc}(B)$-comodule so we have an exact sequence (of $\Oscr_{\tt nc}(B)$-comodules)
$$
0 \to M' \to M \to M/M' \to 0.
$$
By applying $\Hom^{\Oscr_{\tt nc}(B)}(-,t)$, one finds
$$
\Hom^{\Oscr_{\tt nc}(B)}(M/M', t) \cong \Hom^{\Oscr_{\tt nc}(B)}(M,t),
$$ 
and since $M/M'$ has weights $= t$, we see that 
$$
\Hom^{\Oscr_{\tt nc}(B)}(M,t) \cong \Hom^{\Oscr(T)}(M,t).
$$
Now the functor $\Hom^{\Oscr(T)}(-,t)$ is exact since all torus representations are semisimple, i.e. $\Oscr(T)$ is cosemisimple, so indeed we see that $\induced_B^{{\GL_2}}(t)$ is injective in the category of $\Oscr_{\tt nc}({{\GL_2}})$-representations of weight $\leq_2 t$.
\end{proof}
\begin{proof}[Proof of Theorem \ref{th:qh}]
  From their construction as quotients of the $M(\lambda)$, it is
  clear that the~$\nabla_I(\lambda)$,~$\lambda \in \Lambda_n$, are
  defined over $\Oscr_n$. To prove that $\nabla_2(\lambda)=\nabla_I(\lambda)$
we verify the properties (2a),(2b),(2c) in Lemma \ref{lemma:donkincostandard}.

By Proposition~\ref{llll} the subrepresentation $L(\ll)$ cogenerated
by $\lambda$ is the unique simple $\Oscr_{\tt nc}({{\GL_2}})$-subcomodule of
$\nabla_I(\lambda)$, and thus also the unique simple
$\Oscr_n$-subcomodule of $\nabla_I(\ll)$, if $\ll \in \Lambda_n$. This proves \ref{lemma:donkincostandard}(2a). By
Proposition~\ref{prop1}, the $\nabla_I(\ll)$ are injective in the
category of $\Oscr_{\tt nc}({{\GL_2}})$-representations of weights $\leq_2
{\tt wt}(\ll)$. In particular, this shows that for $\ll \in
\Lambda_n$, one has $\tn{Ext}^1_{\Oscr_n}(L(\mu),\nabla_I(\ll))=0$ for
$\mu \leq_2 \ll$. This proves  \ref{lemma:donkincostandard}(2c). Also, all composition factors of $\nabla_I(\ll)$ different from
$L(\lambda)$
are of the form $L(\mu)$ with
 $\mu <_2 \lambda$ since the weights of those composition factors sit
strictly between $\sigma({\tt wt}(\ll))$ and ${\tt wt}(\ll)$ by
Proposition \ref{prop:Ll} and Lemma \ref{lem:bijection}. This proves \ref{lemma:donkincostandard} (2b)
and  we conclude $\nabla_2(\lambda)=\nabla_I(\lambda)$.

To prove that $\Oscr_n$ is quasi-hereditary we verify properties (1)(2) of Proposition \ref{prop:convenient}.
\ref{prop:convenient}(1) follows immediately from Lemma \ref{lemma:filtration}. To prove
\ref{prop:convenient}(2) assume that $\Ext^1(L(\mu),\Delta(\lambda))\neq 0$. Then by the proof
of Proposition \ref{prop1}, $\operatorname{wt}(\mu)>\operatorname{wt}(\lambda)$. In other words $\mu>_2\lambda$.

To prove that $\Delta_2(\lambda)=\Delta_I(\lambda)$ we note that by
definition $\Delta_I(\lambda)=\nabla_I({}^\ast\!\lambda)^\ast$
(${}^\ast\!(-)$ is the inverse to $(-)^\ast$) and by Proposition \label{prop:Ll} we have $L(\lambda)=L({}^\ast\!\lambda)^\ast$.
One now verifies the properties (2a),(2b),(2c) for $\Delta_I(\lambda)$ in the dual version of Lemma \ref{lemma:donkincostandard}
by dualizing the corresponding properties for $\nabla_I(\lambda)$.
\end{proof}

\section{$\Oscr_{\tt nc}({{\GL_2}})$ is quasi-hereditary}

To prove $\Oscr_{\tt nc}({{\GL_2}})$ is quasi-hereditary as in Definition~\ref{strong}, we use a different ordering on $\Lambda$. Note that $\Lambda$ is in a natural way a semigroup. The new ordering is the left-right invariant ordering generated by 
\begin{equation}
\begin{aligned}
1&<d\delta^{-1}d \\ \delta&<dd
\end{aligned}
\end{equation}
and will be denoted $<_1$. This ordering is invariant under $(-)^*$. Note also that for any $\lambda \in \Lambda$, the set $\pi(\ll)=\{\mu \in \Lambda \ \vert \ \mu <_1 \lambda\}$ is finite. So in particular, there do exist finite saturated subsets. 
	\begin{lemma}
\label{lemma:nablafilt}
	For a given $\lambda = \delta^{x_1} d^{y_1} \cdots \delta^{x_n} d^{y_n} \in \Lambda$, $M(\ll)$ has a $\nabla_I$-filtration 
        such that the subquotients are of the form $\nabla_I(\mu)$, for $\mu \in \pi(\ll)$ and $\nabla_I(\lambda)$ occurring with multiplicity one.
	\end{lemma}
	\begin{proof}
	We may prove this by constructing the filtration explicitly. E.g. for $\lambda=d^4$ one has
$$
0 \subset R^2 \subset VVR \subset VVR + VRV \subset VVR+VRV+RVV \subset VVVV=M(\ll),
$$
with respective subquotients $\nabla_I(\delta^2) : \nabla_I(d^2 \delta) : \nabla_I(d \delta d) : \nabla_I(\delta d^2) : 
\nabla_I(d^4)$. If we take $\lambda=d^2\delta^{-1}d$, then one has
$$
0 \subset V \subset VVR^{-1}V,
$$
with subquotients $\nabla_I(d) : \nabla_I(d^2 \delta^{-1}d)$. 
	\end{proof}

We first show that the~$\Oscr_n$ as used in the previous section, are also quasi-hereditary with respect to~$\leq_1$. The corresponding costandard comodules will be denoted~$\nabla_1$.

\begin{lemma}
\label{lemma:previous}
The coalgebra~$\Oscr_n$ fulfills the three conditions of Definition~\ref{qh} with respect to 
$\nabla_I(\lambda), \lambda \in \Lambda_n$. More explicitly:
\begin{enumerate}
\item $I(\lambda) \in \Fscr(\nabla_I),$ 
\item $(I(\lambda):\nabla_I(\lambda))=1,$ 
\item $(I(\lambda):\nabla_I(\mu)) \neq 0 \Rightarrow \mu \geq_1 \lambda.$
\end{enumerate}
\end{lemma}
\begin{proof}
By Theorem \ref{th:qh} we know that~$\Oscr_n$ is quasi-hereditary with
  respect to~$\leq_2$, and~$\nabla_2=\nabla_I$, $\Delta_2=\Delta_I$.
In particular  conditions~(1) and~(2) are satisfied. Condition~(3) requires more
  work, since~$\leq_2$ is a refinement of~$\leq_1$.
We will use that since~$\Oscr_n$ is
quasi-hereditary with respect to~$\leq_2$,  that 
$$
\tn{Ext}^1_{\Oscr_n}(\Delta_I,\nabla_I)=0.
$$
(for example by
Definition~\ref{definition:drqh} below).
 First of all, if
  $\nabla_I(\mu)$ is a $\nabla_I$-composition factor of $I(\lambda)$, then
  there exists a chain $\mu=\lambda_0, \lambda_1, \ldots,
  \lambda_{n-1}, \lambda_n=\lambda$, such that
\begin{align}
\label{ext2}
\tn{Ext}^1(\nabla_I(\lambda_i),\nabla_I(\lambda_{i+1})) \neq 0,
\end{align}
since~$\nabla_I(\lambda)$ is the lowest piece in the~$\nabla$-filtration on~$I(\lambda)$.
The second claim is that whenever one has $\gamma,\eta \in \Lambda$ such that
\begin{align}
\label{ext}
\tn{Ext}^1(\nabla_I(\gamma),\nabla_I(\eta)) \neq 0,
\end{align}
one has $\gamma \geq_1 \eta$. To see this, note that there is an exact sequence
$$
0 \to K \to M(\gamma) \to \nabla_I(\gamma) \to 0.
$$
From the explicit form of the $\nabla_I$-filtration on the $M(\gamma)$ (see Lemma \ref{lemma:nablafilt}), we know that $K$ has a filtration by $\nabla_I(\gamma')$, for $\gamma >_1 \gamma'$. Applying $\Hom(-,\nabla_I(\eta))$, one obtains
$$
\Hom(K,\nabla_I(\eta)) \to \tn{Ext}^1(\nabla_I(\gamma),\nabla_I(\eta)) \to 0,
$$
since $M(\gamma)$ has a $\Delta_I=\Delta_2$-filtration. In
particular,~$\Hom(K,\nabla_I(\eta)) \neq 0$. This implies that
$\Hom(\nabla_I(\zeta),\nabla_I(\eta)) \neq 0$ for some
$\nabla_I$-filtration factor of $K$. Again, we can study this condition
using the $M$'s: from the existence of the surjective map
$M(\zeta) \twoheadrightarrow \nabla_I(\zeta)$, it follows that
$\Hom(M(\zeta),\nabla_I(\eta)) \neq 0$. 

Now since~$M(\zeta)$ also has a filtration by $\Delta_I$'s, one can
explicitly compute the relevant Hom-space from the spaces
$\Hom(\Delta_I(\theta),\nabla_I(\eta))$, for $\Delta_I(\theta)$ a
$\Delta_I$-filtration factor of $M(\zeta)$. By again using the explicit
filtration on the $M$'s, the only way for these Hom-spaces not to
vanish is for $\zeta \geq_1 \eta$, i.e.
\begin{align}
\label{homms}
\Hom(\nabla_I(\zeta),\nabla_I(\eta)) \neq 0 \Rightarrow \zeta \geq_1 \eta. 
\end{align}
Hence we obtain $\gamma \geq_1 \zeta \geq_1 \eta$, so this means that~\eqref{ext} is true, and by~\eqref{ext2} we have $\lambda=\mu_0 \geq_1 \mu_1 \geq_1 \cdots \geq_1 \mu_{n-1} \geq_1 \mu_n=\mu$.
\end{proof}

Note that to prove that~$\Oscr_n$ is quasi-hereditary with respect to~$\leq_1$, we still need to show that~$\nabla_1=\nabla_I$.

\begin{lemma}
\label{lem:previous2}
For~$\Oscr_n$, the comodules~$\nabla_I(\lambda)$ coincide with the costandard comodules~$\nabla_1(\lambda)$ with respect to~$\leq_1$.
\end{lemma}
\begin{proof}
  If we prove that for~$L(\mu)$ a composition factor of
  $\nabla_I(\lambda)/\socle(\nabla_I(\lambda))$, it follows that~$\mu
  <_1 \lambda$, then from Lemma \ref{lemma:donkincostandard} and the
  fact that we already know
  that~$\nabla_I(\lambda)=\nabla_2(\lambda)$ (so that Lemma \ref{lemma:donkincostandard}(2a,2c) hold), we get
  that~$\nabla_I(\lambda)=\nabla_1(\lambda)$.
Since $L(\mu)$ is a composition factor of $\nabla_I(\lambda)$, $\Hom(\nabla_I(\lambda),I(\mu)) \neq 0$. 
We already know from Lemma~\ref{lemma:previous} that all injectives have a~$\nabla_I$-filtration, and the $\nabla_I$-filtration factors of $I(\mu)$ are all $\geq_1 \mu$, so it will suffice to show that 
\begin{align}
\Hom(\nabla_I(\lambda),\nabla_I(\xi)) \neq 0 \Rightarrow \lambda \geq_1 \xi. 
\end{align}
This was already shown during the proof of Lemma~\ref{lemma:previous}, see~\eqref{homms}.
\end{proof}

	\begin{corollary}
	The coalgebra $\Oscr_n$ is quasi-hereditary with respect to the poset $(\Lambda_n,\leq_1)$. 
	Furthermore~$\nabla_1(\lambda)=\nabla_I(\lambda)$, $\Delta_1(\lambda)=\Delta_I(\lambda)$ for~$\lambda \in \Lambda_n$.
	\end{corollary}
	\begin{proof}
	The fact that $\Oscr_n$ is quasi-hereditary and the equality $\nabla_1(\lambda)=\nabla_I(\lambda)$ follow immediately from Lemmas \ref{lemma:previous},\ref{lem:previous2} and
Definition~\ref{qh}. Since $\Oscr_n$ is quasi-hereditary with respect to $\ge_2$ we have by Definition \ref{definition:drqh} below and Proposition \ref{simples}
\[
\Ext^\ast(\Delta_2(\lambda),\nabla_2(\mu))
=
\begin{cases}
k&\text{if $\lambda=\mu$}\\
0&\text{otherwise}
\end{cases}
\]
Since $\Delta_2(\lambda)=\Delta_I(\lambda)$ by Theorem \ref{th:qh} and $\nabla_2(\mu)=\nabla_I(\mu)=\nabla_1(\mu)$ by Theorem \ref{th:qh} and the previous paragraph.
Hence we get
\[
\Ext^\ast(\Delta_I(\lambda),\nabla_1(\mu))
=
\begin{cases}
k&\text{if $\lambda=\mu$}\\
0&\text{otherwise}
\end{cases}
\]
Using the dual version of Definition \ref{definition:drqh} below we deduce easily from this that $\Delta_1(\lambda)=\Delta_I(\lambda)$.
	\end{proof}

\begin{theorem} 
\label{th:qh}
The coalgebra~$\Oscr_{\tt nc}({{\GL_2}})$ is quasi-hereditary with respect to the poset $(\Lambda,\leq_1)$. Furthermore
one has $\nabla_1=\nabla_I$, $\Delta_1=\Delta_I$.
\end{theorem}
\begin{proof} 
  To prove the theorem, we need to check that for every finite
  saturated subset $\pi \subset \Lambda$,  $\Oscr_{\tt nc}({{\GL_2}})(\pi)$, is finite dimensional and
  quasi-hereditary, for the poset $(\pi,\leq_1)$.
Since~$\pi$ is finite, it is clear that~$\Oscr_{\pi}
\subset \Oscr_n$ for some~$n$. It now suffices to invoke Theorem \ref{th:truncation}.
\end{proof}

\begin{corollary} The $M(\lambda)$ are (partial) tilting modules.
\end{corollary}
\begin{proof} This follows in the usual way together with the fact
that by Lemma \ref{lemma:nablafilt} and its dual version the $M(\lambda)$ have both a
$\nabla_I$-filtration and a $\Delta_I$-filtration.
\end{proof}

\begin{corollary} Let $\Rep_{\tt nc}({{\GL_2}})$ be the representation ring of $\Oscr_{\text nc}({{\GL_2}})$. 
There is an isomorphism of rings
\[
\ZZ\langle x,y^{\pm}\rangle\rightarrow \Rep_{\tt nc}({{\GL_2}}):x\mapsto [V], y\mapsto [R]
\]
\end{corollary}
\begin{proof} Using the appropriate infinite dimensional version of Lemma \ref{grothendieck} together with Lemma \ref{lemma:nablafilt} one obtains that $[M(\lambda)]$  for $\lambda\in \Lambda$ is
a basis for $\Rep_{\tt nc}({{\GL_2}})$. It now suffices to note that $M(\lambda_1\lambda_2)=M(\lambda_1)M(\lambda_2)$ and $M(d)=V$, $M(\delta)=R$.
\end{proof}

\section{The simple representations}
\label{section:simples}
Now we assume that $k$ has characteristic zero. From now on we write $\Delta=\Delta_I=\Delta_2=\Delta_1$, $\nabla=\nabla_I=\nabla_2=\nabla_1$.
To study the simple representations of $\Oscr_{\tt nc}({{\GL_2}})$, we use Corollary \ref{simples}, i.e.
$$
L(\ll)=\textnormal{Im}(\Delta(\ll) \to \nabla(\ll)).
$$
Since we already proved that $\Oscr_{\tt nc}({{\GL_2}})$ is quasi-hereditary, the $L(\ll)$ are all the simple representations. We start off by analyzing the map $\Delta(\lambda) \to \nabla(\lambda)$, which was defined as composition 
$$
\Delta(\lambda) \xrightarrow{\pi} M(\lambda) \xrightarrow{\sigma} \nabla(\lambda).
$$
The map $\sigma$ is just the natural quotient map corresponding to 
$$
M(\lambda)=R^{x_1}V^{y_1} \cdots R^{x_n}V^{y_n} \twoheadrightarrow R^{x_1}(S^{y_1}V) \cdots R^{x_n}(S^{y_n}V)=\nabla(\ll),
$$
To understand $\pi$, we first need to understand $\Delta(\lambda)$, which can be accomplished by using the definition from Section~\ref{section:standard}, i.e. first write $\lambda$ as $\mu^*$ and dualize $\nabla(\mu)$. For this, we look at the map
$$
V^y \twoheadrightarrow S^y V,
$$
and dualize, to obtain
$$
(S^y V)^* \hookrightarrow (VR^{-1})^y.
$$
Now define 
$$
T^y(V)\overset{\text{def}}{=}(S^y V)^* R \hookrightarrow VR^{-1}VR^{-1} \cdots VR^{-1} V,
$$
where $V$ appears $y$ times. Then rewrite $M(\lambda)$ as 
$$
R^{s_1}\underbrace{(VR^{-1} \cdots R^{-1} V)}_{t_1 \tn{ times }V} R^{s_2} \underbrace{(VR^{-1} \cdots R^{-1} V)}_{t_2 \tn{ times }V} \cdots R^{s_m}\underbrace{(VR^{-1} \cdots R^{-1} V)}_{t_m \tn{ times }V}, 
$$
for suitable $t_1,t_2,\ldots$. It is easy to see that 
$$
\Delta(\ll)=R^{s_1}(T^{t_1}V) R^{s_2}(T^{t_2}V) \cdots R^{s_m}(T^{t_m}V).
$$
The map $\sigma \circ \pi$ is thus just rewriting words of $M(\ll)$ in a different way. The representation theory of $\GL_2$ yields
$$
T^n V=(S^n V)^* R=S^n V \ot R^{-n+1}
$$ 
(e.g. because both sides are indecomposable representation with the same highest weight). The maps we are thus led to consider are of the form
$$
f_\ll: R^{s_1}(T^{t_1}V) R^{s_2}(T^{t_2}V) \cdots R^{s_m}(T^{t_m}V) \to R^{x_1}(S^{y_1}V) \cdots R^{x_n}(S^{y_n}V).
$$

	\begin{lemma}
	\label{tensor}
	For $\lambda \in \Lambda$ of the (reduced) form $\lambda=\lambda_1 \delta^{i} \lambda_2,$ and $i \neq -1$, 
	one has
	$$
	\image(f_\ll) \cong \image(f_{\lambda_1}) \otimes R^{i} \otimes \image(f_{\lambda_2}).
	$$
	\end{lemma}
	\begin{proof}
	For $\lambda=\delta^{i}$, it is clear that $\image{f_\ll}=R^{i}$. Now suppose $\lambda=\lambda_1 
	\delta^{i} \lambda_2$, then $\nabla(\ll)=\nabla(\ll_1)  R^i  \nabla(\ll_2)$. Now let $
	\mu_1, \mu_2$ be defined by $\lambda_1=\mu_1^*, \lambda_2=
	\mu_2^*$. Then $\lambda=\mu^*\overset{\text{def}}{=}(\mu_2 \delta^{-i} \mu_1)^*$, and since $i \neq 
	-1$, a non-zero factor of this $\delta^{-i}$ will remain in the reduced form of $\mu$. In other words, no higher powers of $d$ will be created in $\mu$ that were not 
	already present in $\mu_1$ or $\mu_2$. This means that
	$$
	\nabla(\mu)=\nabla(\mu_2) R^{-i} \nabla(\mu_1).
	$$ 
	Remember that $\Delta(\lambda)$ was originally defined as $\nabla(\mu)^*$. In this case we get
	$$
	\Delta(\ll)=	\nabla(\mu)^*=\Delta(\lambda_1) R^i \Delta(\lambda_2).
	$$
	Since the tensor products are over $k$, the image of the tensor product is the tensor product of the 
	images, so
	$$
	\image(f_\ll)=\image(\Delta(\lambda_1) R^i \Delta(\lambda_2) \to \nabla(\ll_1)  R^i  \nabla(\ll_2))=\image(f_{\lambda_1})  R^i \image(f_{\lambda_2}).
	$$
	\end{proof}

From now on, we'll only look at $\lambda \in \Lambda$ that do not contain $\delta^i$, for $i \neq -1$, since Lemma~\ref{tensor} shows that the computations for general $\lambda$ can be reduced to this one. Looking at these maps as $\GL_2$-representations, we see that they are compositions of basic maps 
\begin{equation}
\label{compo}
(S^a V)(S^bV) \to (S^{a-1}V)V(S^{b}V) \to (S^{a-1}V)(S^{b+1}V),
\end{equation}
with an obvious definition.
\begin{lemma}
\label{lemma1}
With $V$ denoting the standard representation for $\GL_2$, a $\GL_2$-map of the form
$$
(S^a V)(S^bV) \xrightarrow{f} (S^{a-1}V)(S^{b+1}V)
$$
factorizing as in~\eqref{compo} is always injective or surjective. More precisely, if $a\geq b+1$, then $f$ is injective, if $a \leq b+1$, it is surjective. In particular, if $a=b+1$ then $f$ is a bijection.
\end{lemma}
\begin{proof}
Put $A=k[x_1,x_2,y_1,y_2]=\oplus_{i,j}S^i V \ot S^j V$. Then the map $f$ is the restriction of the $\GL_2$-invariant differential operator on $A$ given by
$$
E=y_1 \frac{\partial}{\partial x_1} + y_2 \frac{\partial}{\partial x_2}.
$$
Now put 
$$
F=x_1 \frac{\partial}{\partial y_1} + x_2 \frac{\partial}{\partial y_2},
$$
such that 
$$
H=[E,F]=\bigg(y_1 \frac{\partial}{\partial y_1} - x_1 \frac{\partial}{\partial x_1}\bigg) + \bigg(y_2 \frac{\partial}{\partial y_2} - x_2 \frac{\partial}{\partial y_2}\bigg).
$$
One can easily check this defines an $\mathfrak{sl}_2$-action on $A$, which is locally finite dimensional, i.e.
$$
A=\oplus_n \oplus_{i+j=n} S^i V \ot S^j V,
$$
and the map $f$ in the statement of the lemma is just the composition
$$
(S^a V)(S^b V) \hookrightarrow \oplus_{i+j=a+b} (S^i V)(S^j V) \xrightarrow{E} \oplus_{i+j=a+b} (S^i V)(S^j V).
$$
Now $E$ acts injectively on the part of $A$ corresponding to strictly negative $H$-eigenvalues, and surjectively for all positive $H$-eigenvalues, since the irreducibles look (up to a shift) like 

$$
\begin{tikzcd} 
\cdots \arrow{r}{E}
\arrow[loop]{r}{H}
    & -2 \arrow{r}{E}
    \arrow[bend left=50]{l}{F}
    \arrow[loop]{r}{H}
    & 0 \arrow[bend left=50]{l}{F}
    \arrow{r}{E}
    \arrow[loop]{rl}{H}
    & + 2 \arrow[bend left=50]{l}{F}
    \arrow{r}{E}
    \arrow[loop]{r}{H}
    & \cdots \arrow[bend left=50]{l}{F}
    \arrow[loop]{r}{H}
\end{tikzcd}
$$

Moreover, the action of $H$ on an element of $(S^a V)(S^b V)$ is just multiplication by $b-a$, so we get that the map is injective if $a>b$ and surjective if $a \leq b$. For $a=b+1$, the dimensions coincide so we have a bijection.
\end{proof}
Let us illustrate the general procedure by means of two examples.
\begin{example}
\label{example}
Let $\lambda=d\delta^{-1}dd$. Then $M(\lambda)=VR^{-1}VV$ and this can be viewed as either $VR^{-1}(VV)$ or as $(VR^{-1}V)V$. We incorporate this into the notation by writing 
$$
\overline{\mkern-48mu\underline{VR^{-1}V}V}
$$
The map we care about is then
$$
(T^2 V) V \to VR^{-1}(S^2 V)
$$
As $\GL_2$-representations, this becomes
$$
(S^2 V) R^{-1} V \to VR^{-1}(S^2 V)
$$
Canceling out the $R$'s, we get the map
$$
(S^2V)V \to V(S^2V),
$$
which just peals off a copy of $V$ on the left and sticks it on the right. More precisely, there is a factorization
$$
(S^2V)V \to VVV \to V(S^2V).
$$
Using the lemma, this map is injective (in fact an isomorphism), so we conclude that $L(\lambda)=(T^2V)V\cong(S^2V)^* R V$. 
\end{example}
\begin{example}
Let $\lambda=d\delta^{-1}d\delta^{-1}d^3$. Then $M(\lambda)$ is given by
$$
\overline{\mkern-95mu\underline{VR^{-1}VR^{-1}V}VV},
$$
so we look at 
$$
(T^3 V)VV \to VR^{-1}VR^{-1}(S^3 V).
$$
This factorizes as
$$
\begin{tikzcd}
	(T^3 V)VV \rar{F} \dar[twoheadrightarrow] & VR^{-1}VR^{-1}(S^3 V) \\
	(T^3 V)(S^2 V) \rar[dashed]{f} & (T^2 V)R^{-1}(S^3V) \uar[hookrightarrow]
\end{tikzcd}
$$
and since the image of $F$ is equal to the image of $f$, we study the map $f$. As $\GL_2$-representations, we see that $f$ becomes 
$$
(S^3 V)(S^2 V) \xrightarrow{f} (S^2V)(S^3V),
$$
again of the form we studied above. This is again injective, so $L(\lambda)=(T^3 V)(S^2V)\cong(S^3V)^* R (S^2 V)$. 
\end{example}
To make the general case manageable, we need the following simple lemma.
\begin{lemma}
\label{lemma2}
If a map 
$$
(S^{a_1} V)(S^{a_2} V)\ldots(S^{a_{2k}}V) \to (S^{a_1 \pm 1}V)(S^{a_2 \mp 2}V)(S^{a_3  \pm 2}V)(S^{a_4 \mp 2}) \ldots (S^{a_{2k} \mp 1}V),
$$
or 
$$
(S^{a_1} V)(S^{a_2} V)\ldots(S^{a_{2k+1}}V) \to (S^{a_1 \pm 1}V)(S^{a_2 \mp 2}V)(S^{a_3  \pm 2}V)(S^{a_4 \mp 2}) \ldots (S^{a_{2k+1} \pm 1}V)
$$
is given by some composition of maps $E_i$ like in Lemma~\ref{lemma1}, the order of composition does not matter.
\end{lemma}
\begin{proof}
This is a straightforward computation. The most interesting case is
\begin{equation}
\label{afbeelding}
(S^a V)(S^b V)(S^c V) \to (S^{a-1}V)(S^{b+2}V)(S^{c-1}V)
\end{equation}
one sets $A=k[x_1,x_2;y_1,y_2;z_1,z_2]=\oplus_{a,b,c} S^a V \ot S^b V \ot S^c V$, 
$$
E_1=y_1 \frac{\partial}{\partial x_1} + y_2 \frac{\partial}{\partial x_2}
$$
$$
E_2=y_1 \frac{\partial}{\partial z_1} + y_2 \frac{\partial}{\partial z_2}.
$$
Then to prove the lemma one needs to check that $E_1$ and $E_2$ commute; this is obvious.
\end{proof}

We now have enough tools to prove the following theorem, where the underlined tensor sign~$\underline{\otimes}$ is multi-valued: it can denote either~$\otimes$ or~$\otimes R^{-1} \otimes$.

\begin{theorem}
\label{th:simplelist}
\begin{enumerate}
\item
	Assume $\lambda \in \Lambda$ does not contain $\delta^i$, for $i \neq -1$.
Then the unique simple 
	representation corresponding to $\ll$ is of the form
	$$
	L(\ll)=T^{a_1}V \underline{\otimes} S^{a_2}V \underline{\otimes} T^{a_3}V \underline{\otimes} \cdots \underline{\otimes} T^{a_n}V,
	$$
	or a similar expression starting and/or ending with $S^a V$. Moreover, in such an expression,
	the exponents of subexpressions have to satisfy certain inequalities:
	$$
	\begin{tabular}{|r|l|}
  	\hline
  	Subexpression & Inequality \\
  	\hline
  	$T^a V $$\ot$$ S^b V$ & $a \geq b+1$ \\
  	$S^b V$$ \ot $$T^a V$ & $a \geq b+1$ \\ 
  	$T^a V$$ \ot R^{-1} \ot$$ S^b V$ & $a+1\leq b$ \\
  	$S^b V$$\otimes R^{-1} \ot$$ T^a V$ & $a+1 \leq b$ \\
  	\hline
	\end{tabular}
	$$
\item
If $\lambda$ is of the form $\lambda_1\delta^i\lambda_2$ with $i\neq 0,-1$ then
\[
L(\ll)=L(\ll_1)\otimes R^i\otimes L(\ll_2)
\]
\end{enumerate}
	\end{theorem}
	\begin{proof}
Number (2) is just a rephrazing of Lemma~\ref{tensor}. For (1), let us first consider the following special situation:
\begin{equation}
\label{voor}
\underbrace{VR^{-1} \cdots VR^{-1}}_{(a-1) \times V} \overbrace{VV\cdots VV}^{(b+2) \times V} \underbrace{R^{-1}V \cdots R^{-1}V}_{(c-1) \times V}.
\end{equation}
The diagram corresponding to this representation is then
$$
\begin{tikzcd}
	(T^a V) V^{b} (T^c V) \rar \dar[twoheadrightarrow] & (VR^{-1})^{a-1} (S^{b+2} V) (R^{-1} V)^{c-1} \\
	(T^a V) (S^{b}V)(T^c V) \rar[dashed] & (T^{a-1}V)R^{-1}(S^b V)R^{-1}(T^{c-1} V) \uar[hookrightarrow]  
\end{tikzcd}
$$
As $\GL_2$-representations, this becomes a map like~\eqref{afbeelding}, i.e.
\begin{equation}
\label{afbeelding2}
(S^a V)(S^b V)(S^c V) \xrightarrow{\delta} (S^{a-1}V)(S^{b+2}V)(S^{c-1}V).
\end{equation}
We cannot directly use Lemma~\ref{lemma1}, because we have three factors. Now~\eqref{afbeelding2} has two possible factorizations:
$$
\begin{tikzcd}
	(S^a V)(S^b V)(S^c V) \rar{f} \dar[red]{E_1} \drar[blue] & (S^{a-1}V)(S^{b+2}V)(S^{c-1}V) \\
	(S^{a-1}V)(S^{b+1}V)(S^{c}V) \urar[red,crossing over]{E_2} & (S^a V)(S^{b+1}V)(S^{c-1}V) \uar[blue]{E_1} 
\end{tikzcd}
$$
and by Lemma~\ref{lemma2}, we know the specific factorization is of no importance. To be able to compute the image (and hence the corresponding simple), we want that in at least one of the two factorizations, there are
\begin{enumerate}
	\item Two surjections
	\item A surjection followed by an injection
	\item Two injections
\end{enumerate}
Indeed, in those cases the images are:
\begin{enumerate}
	\item $\tn{Im}(f)=(S^{a-1}V)(S^{b+2}V)(S^{c-1}V)$
	\item ${\color{red}\tn{Im}(f)}= (S^{a-1}V)(S^{b+1}V)(S^c V)$, ${\color{blue}\tn{Im}(f)}= (S^a V)(S^{b+1}V)(S^{c-1}V)$
	\item $\tn{Im}(f)=(S^{a}V)(S^{b}V)(S^{c}V)$
\end{enumerate}
It remains to check that at least one of the two factorizations falls into one of these three classes. This is a simple numerical check based on Lemma~\ref{lemma1}. 
\begin{remark}
Note that for maps of the form $(S^a V)(S^b V) \to (S^{a+1}V)(S^{b-1}V)$, the inequalities in the lemma have to be reversed, i.e. the map is injective if $a+1 \leq b$ and surjective if $a+1 \geq b$.
\end{remark}

For the general case where we don't have $3$ (different type) factors like in~\eqref{voor}, but any number of them, the corresponding map of $\GL_2$-representations will be a composition of a number of differential operators of the form 
$$
E=y_1 \frac{\partial}{\partial x_1} + y_2 \frac{\partial}{\partial x_2}.
$$
More precisely, we get two of these operators for each factor of type $VR^{-1}VR^{-1} \cdots R^{-1}V$; one of these peels off a copy of $V$ and puts it on the left, and the other one puts it on the right. The $y$-variables correspond to the symmetric power coming from $VV \cdots VV$, and the $x$-variables to the symmetric power coming from $VR^{-1}VR^{-1} \cdots R^{-1}V$. These operators still commute (this is again Lemma~\ref{lemma2}) so we can factor these maps in any way we like as a composition of, say $m$, basic maps. 

A factorization allowing us to compute the image is now one given by $k \geq 0$ consecutive surjections followed by $m-k$ consecutive injections, and comes with a set of inequalities that the exponents of the symmetric powers have to satisfy for it to occur. The corresponding simple representation is then given by the lift of the tensor product of symmetric powers appearing as the codomain of the last surjective map and is thus of the form we want. 

What remains to be checked is that the systems of inequalities cover all occurring cases, i.e. nice factorizations always exist. The $\GL_2$-maps can be represented by exponent tuples as follows
\begin{equation}
\label{vorrr}
(a_1 \vert b_1 \vert \cdots \vert a_n \vert b_n \vert a_{n+1}) \to (a_1 -1 \vert  b_1+2 \vert a_2 -2 \vert b_2+2 \vert \cdots \vert b_n +2 \vert  a_{n+1}-1),
\end{equation}
so there are $2n$ basic maps to start with. Pick a surjective one (possibly bijective), apply it, and keep on applying surjective ones until we are in a situation where all basic maps one can apply are injective (and not surjective). Now keep on applying injective maps. A priori there is the problem that applying a basic map changes an $a_i$, so the algorithm we just described might not end up in the codomain of~\eqref{vorrr}. 

This does not happen, and we will be content with describing a representative example. Look at the map
$$
(a_1 \vert b_1 \vert a_2 \vert b_2 \vert a_3) \to (a_1-1 \vert b_1+2 \vert a_2-2 \vert b_2 +2 \vert a_3 -1),
$$
and suppose all the basic maps are injections (not surjections). In particular we have $a_2 > b_2 +1$. Then by the $3$-factor considerations we made earlier $(a_1 \vert b_1 \vert a_2) \to (a_1-1 \vert b_1+2 \vert a_2-1)$ can be factorized as two injections. One then has to factorize $(a_2-1 \vert b_2 \vert a_3) \to (a_2-2 \vert b_2+2 \vert a_3-1)$, and it could happen, that $a_2-1 \ngtr b_2+1 $. For this however, it is necessary that $a_2=b_2+2$ and thus the corresponding basic map is a bijection, so there is no problem and the algorithm's fine.  
\end{proof}

	\begin{corollary}
\label{cor:simple}
	All simple $\Oscr_{\tt nc}(\GL_2)$-representations are repeated tensor products of simple $\Oscr	
	(M_2)$ representations and their duals. 
	\end{corollary}
	
Here's one more example to clarify the theorem. It gives a situation where $\Delta(\ll) \to \nabla(\ll)$ is neither an epimorphism nor a monomorphism.

	\begin{example}
	Let $\ll=d\delta^{-1}d^4(\delta^{-1}d)^4$. One reduces the problem to computing the image of $f$:
	$$
	\begin{tikzcd}
	(T^2 V)VV(T^5 V) \rar{F} \dar[twoheadrightarrow] & VR^{-1}(S^4 V)(R^{-1}V)^4 \\
	(T^2 V)(S^2 V)(S^5 V) \rar[dashed]{f} & VR^{-1}(S^4 V)R^{-1}(T^4 V) \uar[hookrightarrow]
	\end{tikzcd}
	$$
	As $\GL_2$-representations we have the factorization
	$$
	(S^2 V)(S^2 V)(S^5V) \twoheadrightarrow V(S^3 V)(S^5V) \hookrightarrow V(S^4 V)(S^4V),
	$$
	and neither of the arrows is an isomorphism. Thus, 
	$L(\ll)=VR^{-1}(S^3V)(T^5V)$, so is of the form~$T^1V \underline{\otimes} S^3V \underline{\otimes} T^5V$. 
	\end{example}

	\begin{corollary}
	Amongst the expressions in the theorem, one has the following isomorphisms
	\begin{align*}
	T^aV \ot S^{a-1}V \cong T^{a-1}V \ot R^{-1} \ot S^aV \\
	S^{a-1}V \ot T^aV \cong S^aV \ot R^{-1} \ot T^{a-1}V.
	\end{align*}
	\end{corollary}

The corollary again follows immediately from Lemma~\ref{lemma1}. In particular, this implies that $L(d\delta^{-1}d^2)$ from Example\ref{example} is a tilting object, as the isomorphism implies it is a direct summand of $M(d\delta^{-1}d^2)$. Of course the same holds for $V(T^2V)$. 

\appendix
\section{Comparing the definitions of Dlab-Ringel and Donkin}
\label{appA}
In this paper we use  results from both \cite{dlab-ringel-2} and
\cite{donkin}. As even the basic definitions in those papers are
different, one needs to be  careful transfering results
between them. In this appendix we verify for the benefit of the non-expert reader that
 the two theories are the same.

To be compatible with the rest of the paper we work in the coalgebra setting.
Remember that sending $A$ to $A^\ast$ yields a duality finite dimensional~$k$-algebras and finite
dimensional~$k$-coalgebras.
For a fixed finite dimensional
algebra~$A$, there is an isomorphism between
the categories of finite dimensional right~$A$-modules and finite dimensional
left~$A^*$-comodules which is the identity on the underlying vector
spaces. As before a representation is a finite dimensional comodule.

Let $C$ be a finite dimension coalgebra over $k$.
 Fix a poset~$(\Lambda, \leq)$ such that~$\{L(\lambda) \vert \lambda
 \in \Lambda\}$ is a complete set of non-zero, pairwise non-isomorphic
 simple~$C$-comodules. Let~$I(\lambda)$ denote the injective hull
 of~$L(\lambda)$ and let~$P(\lambda)$ denote its projective cover,
 which exists since~$C$ is finite dimensional. The multiplicity of a
 simple comodule~$L(\lambda)$ as a composition factor of the
 representation~$V$ will be denoted~$[V:L(\lambda)]$.

For~$\pi \subset \Lambda$, and~$V$ a~$C$-representation, denote
by~$O_{\pi}(V)$ the unique maximal
subcomodule of~$V$ that has all composition factors indexed by
elements of~$\pi$. Dually~$O^{\pi}(V)$ is the unique minimal
subcomodule~$U$ of~$V$ such that~$V/U$ has
all composition factors indexed by elements of~$\pi$. 
For any~$\lambda \in \Lambda$, set
\begin{equation}
\begin{aligned}
\pi^<(\lambda) &= \{\mu \in \Lambda \vert  \mu < \lambda\} \\
\pi^{\leq}(\lambda) &= \{\mu \in \Lambda \vert \mu \leq \lambda\}.
\end{aligned}
\end{equation}
\subsection{Donkin quasi-hereditary coalgebras} Here we give the
definitions used by Donkin in~\cite{donkin}.  For clarity we will
decorate notations with a subscript ``$D$'' (no such subscript was
used in the body of the paper). Also for clarity we will repeat some
definitions and results already stated in the main text.
\begin{definition}
\label{definition:dcostandard}
The comodule~$\nabla_{\tt D}(\lambda)$ is defined as the unique subcomodule of~$I(\lambda)$ containing~$L(\lambda)$ such that
\begin{equation}
\nabla_{\tt D}(\lambda)/L(\lambda)=O_{\pi^<(\lambda)}(I(\lambda)/L(\lambda)).
\end{equation}
\end{definition}
Denote by~$N(\lambda)$ the maximal strict subcomodule of~$P(\lambda)$.
\begin{definition}
\label{definition:dstandard}
The comodule~$\Delta_{\tt D}(\lambda)$ is defined  as
\begin{equation}
\Delta_{\tt D}(\lambda)=P(\lambda)/O^{\pi^<(\lambda)}(N(\lambda)).
\end{equation}
\end{definition}
Comodules isomorphic to $\Delta_{\tt D}(\lambda)$, $\nabla_{\tt D}(\lambda)$ will be called D(onkin)-standard and costandard comodules respectively. By
the triangular nature of the definition we see that both $[\Delta_{\tt D}(\lambda)]$ and $[\nabla_{\tt D}(\lambda)]$ yield bases of $K_0(C)$.
\begin{example}
Consider the algebra~$A=k[x]/(x^2)$. Then~$\Lambda=\{1\}$,~$P(1)=I(1)=A$, and~$L(1)=k[x]/(x)$. One checks immediately that~$\nabla_{\tt D}(1)=\Delta_{\tt D}(1)=L(1)$.
\end{example}
The costandard comodules can be characterised as follows. 
\begin{lemma}
For any~$C$-comodule~$V$, and~$\lambda \in \Lambda$, the following are equivalent:
\begin{enumerate}
\item $V \cong \nabla_{\tt D}(\lambda)$,
\item the following three conditions are satisfied:
\begin{enumerate}
\item $\socle(V) \cong L(\lambda)$,
\item if~$[V/\socle(V):L(\mu)] \neq 0$, then~$\mu < \lambda$,
\item if~$\mu < \lambda$, then $\Ext^1(L(\mu),V)=0$.
\end{enumerate}
\end{enumerate}
\end{lemma}
Denote by~$\Fscr(\nabla_{\tt D})$, respectively~$\Fscr(\Delta_{\tt D})$, the categories of~$C$-representations that have a filtration with costandard, respectively standard, subquotients. We will call these (co)standard filtrations. Denote by~$(V:\nabla_{\tt D}(\lambda))$ the multiplicity of~$\nabla_{\tt D}(\lambda)$ in a costandard filtration of~$V \in \Fscr(\nabla_{\tt D})$. This number is independent of the filtration
by the fact that costandard comodules form a basis for $K_0(C)$ (see above).
\begin{definition}
\label{definition:dqh}
The coalgebra~$C$ is D(onkin)-quasi-hereditary if for $\lambda$ in $\Lambda$
\begin{enumerate}
\item $I(\lambda) \in \Fscr(\nabla_{\tt D})$,
\item $(I(\lambda):\nabla_{\tt D}(\lambda))=1$,
\item if $(I(\lambda):\nabla_{\tt D}(\mu))\neq 0$ then~$\mu \geq \lambda$.
\end{enumerate}
\end{definition}
\subsection{Dlab-Ringel quasi-hereditary coalgebras} All definitions and results in~\cite{dlab-ringel-2} are for finite dimensional algebras over an arbitrary field, which we transpose to finite dimensional coalgebras over~$k$.
\begin{definition}
The comodule~$\nabla_{\tt DR}(\lambda)$ is defined as the maximal subcomodule of~$I(\lambda)$ with composition factors~$L(\mu)$, for~$\mu \leq \lambda$. In~$O$-notation, this becomes
\begin{equation}
\nabla_{\tt DR}(\lambda)=O_{\pi^{\leq}(\lambda)}(I(\lambda)).
\end{equation}
\end{definition}
\begin{definition}
The comodule~$\Delta_{\tt DR}(\lambda)$ is defined as the maximal factor comodule of~$P(\lambda)$ with composition factors~$L(\mu)$, for~$\mu \leq \lambda$. In~$O$-notation, this becomes
\begin{equation}
\Delta_{\tt DR}(\lambda)=P(\lambda)/O^{\pi^{\leq}(\lambda)}(P(\lambda)).
\end{equation}
\end{definition}
These comodules (up to isomorphism) are the D(lab)-R(ingel)-(co)standard comodules.
\begin{example}
Consider~$A=k[x]/(x^2)$ again. Then~$\nabla_{\tt DR}(1)=\Delta_{\tt DR}(1)=A$, so these differ from~$\nabla_{\tt D}(1)$ and~$\Delta_{\tt D}(1)$.
\end{example}
The DR-costandard comodules may be characterized as follows.
\begin{lemma}
\label{lemma:costandard}
For any~$C$-representation~$V$, and~$\lambda \in \Lambda$, the following are equivalent:
\begin{enumerate}
\item $V \cong \nabla_{\tt DR}(\lambda)$,
\item the following three conditions are satisfied:
\begin{enumerate}
\item $\socle(V) \cong L(\lambda)$,
\item if~$[V:L(\mu)] \neq 0$, then~$\mu \leq \lambda$,
\item if~$\mu \leq \lambda$, then $\Ext^1(L(\mu),V)=0$.
\end{enumerate}
\end{enumerate}
\end{lemma}
The categories~$\Fscr(\nabla_{\tt DR})$ and~$\Fscr(\Delta_{\tt DR})$ are defined as in the Donkin case. To define quasi-hereditary coalgebras, Dlab and Ringel require~$(\Lambda,\leq)$ to satisfy an additional property, ensuring that the standard and costandard comodules don't change under refinement of this partial order.
\begin{definition}
The poset~$(\Lambda,\leq)$ is said to be adapted if for every~$C$-representation~$V$, with~$\mathrm{top}(V) \cong L(\lambda_1)$, and~$\socle(V)\cong L(\lambda_2)$, such that~$\lambda_1$ and~$\lambda_2$ are incomparable with respect to~$\leq$, there exists a~$\mu$ such that~$\mu > \lambda_1$ and~$\mu > \lambda_2$ such that~$[V:L(\mu)] \neq 0$.
\end{definition}
In fact, they show that a weaker condition suffices to have an adapted ordering.
\begin{lemma}
\label{lemma:adapted}
Suppose that for every~$C$-representation~$V$, with~$\mathrm{top}(V) \cong L(\lambda_1)$, and~$\socle(V)\cong L(\lambda_2)$, such that~$\lambda_1$ and~$\lambda_2$ are incomparable with respect to~$\leq$, there exists a~$\mu$ such that~$\mu > \lambda_1$ or~$\mu > \lambda_2$ such that~$[V:L(\mu)] \neq 0$. Then~$(\Lambda,\leq)$ is adapted.
\end{lemma}
Remember that a representation~$V$ is called Schurian if its endomorphism ring is a division ring. 
\begin{definition}
\label{definition:drqh} 
The coalgebra~$C$ is D(lab)-R(ingel)-quasi-hereditary if~$(\Lambda,\leq)$ is adapted, all costandard comodules~$\nabla_{\tt DR}(\lambda)$ are Schurian, and if one of the following equivalent conditions hold:
\begin{enumerate}
\item ${}^C C \in \Fscr(\nabla_{\text{DR}})$,
\item $\Fscr(\nabla_{\text{DR}})=\{V \vert \Ext^1(\Delta_{\text{DR}},V)=0\}$,
\item $\Fscr(\nabla_{\text{DR}})=\{V \vert \Ext^i(\Delta_{\text{DR}},V)=0 \text{ for all } i \geq 1\}$,
\item $\Ext^2(\Delta_{\text{DR}},\nabla_{\text{DR}})=0$.
\end{enumerate}
\end{definition}
The equivalence of these conditions can be found as Theorem 1 in~\cite{dlab-ringel-2}. Note that by the autoduality of the criteria we obtain
\begin{lemma}
\label{lem:schur1}
It the coalgebra~$C$ is DR-quasi-hereditary then for all~$\lambda \in
\Lambda$, $\nabla_{\tt DR}(\lambda)$ and~$\Delta_{\tt DR}(\lambda)$
are Schurian.
\end{lemma}
\subsection{Equivalence of the definitions} 
In this section we prove the following result.
\begin{theorem}
\label{theorem:equivalence}
The coalgebra~$C$ is D-quasi-hereditary with respect to $(\Lambda,\leq)$ if and only if it is DR-quasi-hereditary with respect to $(\Lambda,\leq)$. Moreover,
in that case $\nabla_{\tt D}(\lambda)=\nabla_{\tt DR}(\lambda)$ and~$\Delta_{\tt D}(\lambda)=\Delta_{\tt DR}(\lambda)$ for all~$\lambda \in \Lambda$.
\end{theorem}
It will be convenient to make the following definition.
\begin{definition}
\label{definition:adapted}
A representation~$V$ is called strongly costandard if the following three conditions are satisfied for some~$\lambda \in \Lambda$:
\begin{enumerate}
\item $\socle(V) \cong L(\lambda)$,
\item if~$[V:L(\mu)] \neq 0$, then~$\mu \leq \lambda$,
\item if $\Ext^1(L(\mu),V)\neq 0$ then $\mu > \lambda$, 
\end{enumerate}
\end{definition}
Note that a strongly costandard comodule is automatically DR-costandard. So it is of the form $\nabla_{\text{DR}}(\lambda)$.
\begin{lemma}
\label{lemma:adapted}
If~$C$ is D-quasi-hereditary, then for every $\lambda\in \Lambda$ one has $\nabla_{\tt D}(\lambda)=\nabla_{\tt DR}(\lambda)$ and moreover
these representations are strongly costandard.
\end{lemma}
\begin{proof}
We want to show that~$\nabla_{\tt D}(\lambda)$ satisfies the conditions stated in Definition \ref{definition:adapted}.
The only condition that is not clear is condition~$(3)$. Suppose $X$ is the middle term of a non-split extension in $\Ext^1(L(\mu), \nabla_{\tt D}(\lambda))$. Using injectivity of~$I(\lambda)$ we get a map~$f$ as follows:
\begin{equation}
\begin{tikzcd}
0 \arrow{r} & \nabla_{\tt D}(\lambda) \arrow{r} \dar & X \arrow{r} \dlar[dashed]{f} & L(\mu) \arrow{r} & 0 \\
 & I(\lambda) & & &
\end{tikzcd}
\end{equation}
Since the extension is non-split,~$\socle(\nabla_{\tt D}(\lambda))=\socle(X)=L(\lambda)$, so~$f$ is injective when restricted to the socles, which implies that~$f:X \to I(\lambda)$ is injective. In particular, there is an injection
\begin{equation}
L(\mu) \cong X/\nabla_{\tt D}(\lambda) \hookrightarrow I(\lambda)/\nabla_{\tt D}(\lambda),
\end{equation} 
from which one gets by Definition \ref{definition:dqh}(3) 
that $L(\mu)$ is contained in some $\nabla_D(\gamma)$ for
$\gamma>\lambda$. Since $\nabla_D(\gamma)\subset I(\gamma)$ has simple socle this is only possible if $\mu=\gamma$ and hence $\mu>\lambda$.
\end{proof}
\begin{lemma}
\label{lemma:schurian}
For all~$\lambda \in \Lambda$, $\nabla_{\tt D}(\lambda)$ and~$\Delta_{\tt D}(\lambda)$ are Schurian. 
\end{lemma}
\begin{proof}
  From Definitions~\ref{definition:dcostandard}
  and~\ref{definition:dstandard} it is immediate that~$[\nabla_{\tt
   D}(\lambda):L(\lambda)]=1$. It follows that any~$f
  \in \End(\Delta_{\tt D}(\lambda))$ induces a morphism
  on~$L(\lambda)$, which is multiplication by a scalar $\alpha$.
 This means that~$f-\alpha$ maps~$\Delta_{\tt
    D}(\lambda)$ into its unique maximal subcomodule~$M$,
  so~$L(\lambda)$ has to be a composition factor of~$\ker(f-\alpha)$. It
  follows that~$\ker(f-\alpha)$ cannot be contained in~$M$, so it has to
  be all of~$\Delta_{\tt D}(\lambda)$, proving the lemma. 
  For~$\nabla_{\tt D}(\lambda)$ the proof is analogous.
\end{proof}
\begin{lemma}
\label{corollary:compatible}
If~$C$ is D-quasi-hereditary or DR-quasi-hereditary, both with respect to~$(\Lambda,\leq)$, then~$\nabla_{\tt DR}(\lambda)=\nabla_{\tt D}(\lambda)$ and~$\Delta_{\tt D}(\lambda)=\Delta_{\tt DR}(\lambda)$.
\end{lemma}
\begin{proof}
If $C$ is $D$-quasi-hereditary then this follows from Lemma \ref{lemma:adapted} and its dual version.

So assume that $C$ is DR-quasi-hereditary.
By Lemmas \ref{lem:schur1}
the (co)standard comodules are Schurian. It is also easily seen that this implies that~$[\nabla_{\tt DR}(\lambda):L(\lambda)]=[\Delta_{\tt DR}(\lambda):L(\lambda)]=1$, and the corollary now follows immediately from the definitions of the respective (co)standard comodules.
\end{proof}

In order to prove Theorem \ref{theorem:equivalence} we also need to address the adapted ordering. 
\begin{lemma}
\label{lem:adapted}
If every DR-costandard comodule is strongly costandard then~$(\Lambda,\leq)$ is adapted.
\end{lemma}
\begin{proof}
Suppose the ordering is not adapted. Then there exists a representation~$V$ with~$\topp(V)\cong L(\mu_1)$ and $\socle(V)\cong L(\mu_2)$, with~$\mu_1$ and~$\mu_2$ incomparable such that if~$[V:L(\mu)] \neq 0$, then~$\mu \ngtr \mu_2$. Condition~$(3)$ in Definition~\ref{definition:adapted} says that~$\nabla_{\tt DR}(\mu_2)$ is injective in the category of representations with composition factors~$\ngtr \mu_2$. Since~$V$ is in this category we get a commuting triangle
\begin{equation}
\begin{tikzcd}
V \arrow{r}{f} & \nabla_{\tt DR}(\mu_2) \\
L(\mu_2) \arrow[hookrightarrow]{u} \urar[hookrightarrow] & 
\end{tikzcd}
\end{equation}
Moreover,~$f$ is injective, since the induced map on the socles is
injective. But now we obtain a contradiction since all composition
factors of~$\nabla_{\tt DR}(\lambda)$ are~$\leq \mu_2$, whereas~$V$
has composition factor~$L(\mu_1)$, and~$\mu_1$ is incomparable
to~$\mu_2$.
\end{proof}
\begin{corollary}
\label{corollary:adapted}
If~$C$ is D-quasi-hereditary, then the poset~$(\Lambda,\leq)$ is adapted.
\end{corollary}
\begin{proof}
This follows from Lemmas \ref{lemma:adapted} and \ref{lem:adapted}.
\end{proof}
We need the following lemma, which can be found as Lemma~$1.3$ in~\cite{dlab-ringel-2}.
\begin{lemma}
\label{lemma:exts}
For~$\lambda, \mu \in \Lambda$, with~$(\Lambda,\leq)$ adapted, one has
\begin{equation}
\Ext^1(\nabla_{\text{DR}}(\lambda),\nabla_{\text{DR}}(\mu)) \neq 0 \Rightarrow \lambda > \mu.
\end{equation}
\end{lemma}
Now we are ready to prove Theorem \ref{theorem:equivalence}.
\begin{proof}[Proof of Theorem \ref{theorem:equivalence}]
  The second part of the theorem is just
  Lemma~\ref{corollary:compatible}. Let us show the first part. Assume
  that~$C$ is D-quasi-hereditary. By Lemma~\ref{lemma:schurian},
  the~$\nabla_{\tt D}(\lambda)=\nabla_{\tt DR}(\lambda)$ are Schurian. Also, by
  Corollary~\ref{corollary:adapted}, the poset~$(\Lambda,\leq)$ is
  adapted. Since by hypothesis all the~$I(\lambda)$ have a filtration by costandard
  comodules, and any coalgebra is the direct sum of its injective
  indecomposables, we get~${}^C C \in \Fscr(\nabla_{\tt
    D})=\Fscr(\nabla_{\tt DR})$, so by definition of
  DR-quasi-hereditary, we are done.

Now assume that~$C$ is DR-quasi-hereditary. By the second equivalent condition in Definition~\ref{definition:drqh}, we get~$I(\lambda) \in \Fscr(\nabla_{\tt DR})$, for all~$\lambda \in \Lambda$. Now assume that
\begin{equation}
\label{eq:globalfilt}
0=F_0 \subset F_1 \subset \ldots \subset F_{k-1} \subset F_k=I(\lambda)
\end{equation}
is a~$\nabla_{\tt DR}$-filtration, i.e.~$F_{p_F}/F_{{p_F}-1} \cong \nabla_{\tt DR}(\lambda_{p_F})$. Since $F_1$ must contain the socle of $I(\lambda)$ we find $F_1=\nabla_{\tt DR}(\lambda)$. 
Assume there exists $j>1$ such that $\lambda_j\not>\lambda$
and let ${p_F}$ be the smallest such $j$. Moreover assume that the filtration \eqref{eq:globalfilt} is chosen with minimal ${p_F}$, among all such filtrations.
If ${p_F}>2$ then $\lambda_{{p_F}-1}>\lambda$ and hence $\lambda_{{p_F}}\not> \lambda_{{p_F}-1}$. The same conclusion holds if ${p_F}=2$ since then $\lambda_{{p_F}-1}=\lambda$.
Then by~Lemma~\ref{lemma:exts}, it follows that the exact sequence
\[
0\rightarrow F_{{p_F}-1}/F_{{p_F}-2}\rightarrow F_{{p_F}}/F_{{p_F}-2}\rightarrow F_{{p_F}}/F_{{p_F}-1}\rightarrow 0
\]
is split. Using this we may create a new filtration
\begin{equation}
0=F_0' \subset F_1' \subset \ldots \subset F_{k-1}' \subset F_k'=I(\lambda)
\end{equation}
such that
\[
F'_j/F_{j-1}'
=
\begin{cases}
F_j/F_{j-1}&\text{if $j\neq {p_F},{p_F}-1$}\\
F_{p_F}/F_{{p_F}-1}&\text{if $j={p_F}-1$}\\
F_{{p_F}-1}/F_{{p_F}}&\text{if $j={p_F}$}\\
\end{cases}
\]
We find $p_{F'}=p_{F}-1$, contradicting the minimality of $p_F$.
\end{proof}
\subsection{Proof of Proposition \ref{prop:convenient}}
Assume first the $C$ is D-quasi-hereditary.  Condition (1) in the statement of Proposition \ref{prop:convenient}  then follows from Definition \ref{definition:dqh}(1) since
${}_C C$ is injective. Condition (2) follows from Lemma \ref{lemma:adapted}.

Conversely assume that Conditions (1)(2) are satisfied. We claim that $C$ is DR-quasi-hereditary, which by Theorem \ref{theorem:equivalence} is sufficient.  We first claim
$\nabla_{\text{DR}}(\lambda)=\nabla_{\text{D}}(\lambda)$. We have in any case $\nabla_{\text{D}}(\lambda)\subset \nabla_{\text{DR}}(\lambda)$. If this not equality then
we must have $\Ext^1(L(\lambda),\nabla_D(\lambda))\neq 0$. By Condition (2) this implies $\lambda>\lambda$ which is a contradiction.

Since $\nabla_{\text{DR}}(\lambda)=\nabla_{\text{D}}(\lambda)$ we obtain from Lemma \ref{lemma:schurian} that $\nabla_{\text{DR}}(\lambda)$ is Schurian. Moreover the ordering
is adapted by Lemma \ref{lem:adapted}. This finishes the proof.


\begin{thebibliography}{10}

\bibitem{AS}
M.~Artin and W.~Schelter, \emph{Graded algebras of global dimension 3}, Adv. in
  Math. \textbf{66} (1987), 171--216.

\bibitem{bichon-dubois-violette}
J.~Bichon and M.~Dubois-Violette, \emph{The quantum group of a preregular
  multilinear form}, Lett. Math. Phys. \textbf{103} (2013), no.~4, 455--468.

\bibitem{bichon-riche}
J.~Bichon and S.~Riche, \emph{Hopf algebras having a dense big cell}, Transactions of the American Mathematical Society, to appear, arXiv:1307.3567 [math.QA].

\bibitem{chevrov-falqui-rubtsov}
A.~Chervov, G.~Falqui, and V.~Rubtsov, \emph{Algebraic properties of {M}anin
  matrices. {I}}, Adv. in Appl. Math. \textbf{43} (2009), 239--315.

\bibitem{chirvasitu}
A.~Chirvasitu, \emph{Grothendieck rings of universal quantum groups}, J.
  Algebra \textbf{349} (2012), 80--97.

\bibitem{dlab-ringel-2}
V.~Dlab and C.M. Ringel, \emph{The module theoretical approach to
  quasi-hereditary algebras}, Representations of Algebras and Related Topics,
  London Math. Soc. Lecture Note Series, vol. 168, Cambridge University Press,
  Cambridge, 1992, pp.~200--224.

\bibitem{donkin2}
S.~Donkin, \emph{The {$q$}-{S}chur algebra}, London Mathematical Society
  Lecture Note Series, vol. 253, Cambridge University Press, Cambridge, 1998.

\bibitem{donkin}
\bysame, \emph{Tilting modules for algebraic groups and finite dimensional
  algebras}, Handbook of tilting theory, London Math. Soc. Lecture Note Ser.,
  vol. 332, Cambridge Univ. Press, Cambridge, 2007, pp.~215--257.

\bibitem{green}
J.A. Green, \emph{Locally finite representations}, J. Algebra \textbf{41}
  (1976), 137--171.

\bibitem{jantzen2}
J.C. Jantzen, \emph{Representations of algebraic groups}, second ed.,
  Mathematical Surveys and Monographs, vol. 107, American Mathematical Society,
  Providence, RI, 2003.

\bibitem{kriegk-vandenbergh}
B.~Kriegk and M.~Van~den Bergh, \emph{Representations of noncommutative quantum
  groups}, Proc. Lond. Math. Soc. (3) \textbf{110} (2015), no.~1, 57--82.

\bibitem{manin}
Y.~I. Manin, \emph{Quantum groups and non-commutative geometry}, Tech. report,
  Centre de Recherches Math\'ematiques, Universit\'e de Montreal, 1988.

\bibitem{pareigis-2}
B.~Pareigis, \emph{Lectures on quantum groups and noncommutative geometry},
  Lecture Notes TU Munich, URL:
  http://www.mathematik.uni-muenchen.de/\~{}pareigis/Vor\-les\-ungen/02SS/QGand\-NCG.pdf.

\bibitem{raedschelders-vandenbergh-2}
T.~Raedschelders and M.~Van den Bergh, \emph{The {M}anin {H}opf algebra of a
  {K}oszul {A}rtin-{S}chelter regular algebra is quasi-hereditary}, arXiv:1509.03157 [math.RT].

\bibitem{takeuchi}
M.~Takeuchi, \emph{Free {H}opf algebras generated by coalgebras}, J. Math. Soc.
  Japan \textbf{23} (1971), 561--582.

\bibitem{walton-wang}
C.~Walton and X.~Wang, \emph{On quantum groups associated to non-noetherian
  regular algebras of dimension 2}, arXiv:1503.09185 [math.RA].
  
\bibitem{wang}
S.~Wang, \emph{Problems in the theory of quantum groups}, Banach Center Publ. \textbf{40} (1997), 67--78.

\end{thebibliography}
\providecommand{\bysame}{\leavevmode\hbox to3em{\hrulefill}\thinspace}
\providecommand{\MR}{\relax\ifhmode\unskip\space\fi MR }
\providecommand{\MRhref}[2]{%
  \href{http://www.ams.org/mathscinet-getitem?mr=#1}{#2}
}
\providecommand{\href}[2]{#2}

\end{document}